\documentclass[11pt]{amsart}
\usepackage[hmargin=2.6cm,bmargin=2.55cm,tmargin=3.05cm]{geometry}
\usepackage[usenames]{color}
\usepackage{graphbox}
\usepackage{subcaption}

\usepackage{mathtools}
\numberwithin{equation}{section}
\mathtoolsset{showonlyrefs,showmanualtags}

\allowdisplaybreaks

\usepackage{amsmath,amssymb,amsthm,amsfonts,enumerate,url,mathbbol,mathrsfs,esint,tikz,graphicx,pifont,float}
\usetikzlibrary{calc}
\usepackage{hyperref}

\usepackage[normalem]{ulem}
\usepackage[utf8]{inputenc}

\theoremstyle{plain}
\newtheorem{theorem}{Theorem}[section]
\newtheorem{lemma}[theorem]{Lemma}
\newtheorem{proposition}[theorem]{Proposition}

\theoremstyle{definition}

\newtheorem{example}[theorem]{Example}

\newtheorem{assumption}[theorem]{Assumption}
\newtheorem{remark}[theorem]{Remark}

\numberwithin{equation}{section}
\numberwithin{figure}{section}

\DeclareMathOperator{\dist}{dist}

\newcommand{\R}{\mathbb{R}}
\newcommand{\N}{\mathbb{N}}

\newcommand{\txtr}[1]{{#1}} 

\newcommand{\partition}{\mathcal P}
\newcommand{\partitionsetk}{\mathfrak P_k}
\newcommand{\partitionset}[1]{\mathfrak P_{#1}}
\newcommand{\optenergy}[2]{\mathcal L_{#1,#2}}
\newcommand{\energy}[2]{\Lambda_{#1,#2}}
\newcommand{\energyrel}[2]{\widetilde\Lambda_{#1,#2}}
\newcommand{\optenergyrel}[2]{\widetilde{\mathcal L}_{#1,#2}}
\newcommand{\threshold}[2]{\mathcal{T}_{#1,#2}}
\newcommand{\thresholdrel}[2]{\widetilde{\mathcal{T}}_{#1,#2}}

\newcommand{\infess}{\Sigma}

\newcommand{\infspec}{\lambda}

\title{Spectral minimal partitions of unbounded domains}
\author{Matthias Hofmann}
\author{James B.\ Kennedy}
\author{Hugo Tavares}

\address{Matthias Hofmann\\Fakult\"at f\"ur Mathematik und Informatik, FernUniversit\"at in Hagen, 58084 Hagen, Germany}
\email{matthias.hofmann@fernuni-hagen.de}

\address{James B. Kennedy\\ Departamento de Matem\'atica, Universidade de Aveiro, 3810-193 Aveiro, Portugal}

\email{jbkennedy@ua.pt}

\address{Hugo Tavares\\CAMGSD -- Centro de An\'alise Matem\'atica, Geometria e Sistemas Din\^amicos, Departamento de Matem\'atica, Instituto Superior T\'ecnico, Av.\ Rovisco Pais, 1049-001 Lisboa, Portugal} 
\email{hugo.n.tavares@tecnico.ulisboa.pt}



\thanks{The authors would like to thank the referee for his/her useful comments on a previous version of this paper. This work was supported by the Funda\c{c}\~ao para a Ci\^encia e a Tecnologia (FCT), Portugal, within the scope of the projects SpectralOPs - Spectral Optimal Partitions: geometric and numerical analysis, reference \href{https://doi.org/10.54499/2023.13921.PEX                                            }{2023.13921.PEX} (all authors), project SHADE - Sharp inequalities and critical problems in harmonic analysis, dispersive and elliptic equations, reference \href{https://doi.org/10.54499/2023.17881.ICDT}{2023.17881.ICDT}  (H.T.), and via the research centers GFM - projects \href{https://doi.org/10.54499/UID/00208/2025}{UID/00208/2025} and \href{https://doi.org/10.54499/UID/PRR/00208/2025}{UID/PRR/00208/2025} (J.B.K.) and CIDMA - projects \href{https://doi.org/10.54499/UID/04106/2025}{UID/04106/2025} and \href{https://doi.org/10.54499/UID/PRR/04106/2025}{UID/PRR/04106/2025} (J.B.K.) and CAMGSD - project \href{https://doi.org/10.54499/UID/04459/2025}{UID/04459/2025} - IST-ID (H.T.) under the FCT Multi-Annual Financing Program for R{\&}D Units. 
}

\begin{document}

\begin{abstract}
We study the problem of constructing $k$-spectral minimal partitions of domains in $d$ dimensions, where the energy functional to be minimized is a $p$-norm ($1 \le p \le \infty$) of the infimum of the spectrum of a suitable Schr\"odinger operator $-\Delta +V$, with Dirichlet conditions on the boundary of the partition elements (cells).  The main novelty of this paper is that the domains may be unbounded, including of infinite volume.  

First, we prove a sharp upper bound for the infimal energy among all $k$-partitions by a threshold value which involves the infimum $\Sigma$ of the essential spectrum of the Schr\"odinger operator on the whole domain as well as the infimal energy among all $k-1$-partitions. Strictly below such threshold, we develop a concentration-compactness-type argument showing optimal partitions exist, and each cell admits ground states (i.e., the infimum of the spectrum on each cell is a simple isolated eigenvalue).

Second, for $p<\infty$, when the energy and the threshold level coincide, we show there may or may not be minimizing partitions. Moreover, even when these exist, they may not have ground states.

Third, for $p=\infty$, minimal partitions always exist, even at the threshold level, but these may or may not admit ground states. Moreover, below the threshold, we can always construct a minimizer, which is an equipartition. At the threshold value we show that spectral minimal partitions may not need to be equipartitions.

We give a variety of examples of both domains and potentials to illustrate the new phenomena that occur in this setting.
\end{abstract}

\maketitle

\section{Introduction}
\label{sec:intro}

By now, there is a well-developed theory of so-called \emph{spectral minimal partitions} (SMPs) of \emph{bounded} Euclidean domains. Prototypically, given a bounded domain $\Omega \subset \R^d$, $d \geq 2$, and a \emph{$k$-partition} of $\Omega$, that is, a collection $\partition = (\omega_1, \ldots, \omega_k)$ of $k$ open and pairwise disjoint subsets $\omega_1,\ldots,\omega_k$ of $\Omega$, on each $\omega_i$ one considers the first eigenvalue $\lambda_1 (\omega_i)$ of the Dirichlet Laplacian, and constructs an \emph{energy functional} of the partition as some $p$-norm ($1 \le p \le \infty$) of the mentioned eigenvalues, namely
\begin{equation}\label{eq:partition_bounded}
    \partition\mapsto \left(\sum_{i=1}^{k} \lambda_1 (\omega_i)^p\right)^{1/p} \quad (1\leq p<\infty),\qquad \partition\mapsto  \max_{i=1,\ldots,k} \{\lambda_1 (\omega_i)\} \quad (p=\infty).
\end{equation}
SMPs are then any partitions minimizing these kinds of energy functional.

Since, very roughly, $\lambda_1 (\omega_i)$ is smaller for larger, ``rounder'' $\omega_i$ (see \cite{BrascoDePhilippisVelichov_FKQ}), such an SMP represents an analytically optimal partition of $\Omega$ into ``large, round'' pieces, or cells, in some sense. However, while a discrete analogue of this principle involving graph Laplacians is one method used to partition discrete graphs into ``clusters'' (\cite{OstWhiOud14}), the main reasons for the interest in such SMPs over the last couple of decades is arguably quite different: first and foremost, due to their connections to the Dirichlet Laplacian on the whole domain, in particular as regards the \emph{nodal partitions}, that is, partitions of $\Omega$ into the nodal domains of its Laplacian eigenfunctions. See \cite{BNHe17} for a detailed discussion, and check the references therein. However, they are also of interest for their links to harmonic maps with values in singular spaces \cite{CaffarelliLin,CaffarelliLinStrat,OgnibeneVelichkov2, TavaresTerracini}; moreover, these or related optimal partition problems are also connected with systems of elliptic equations, in particular as they correspond to the limits of singularly perturbed elliptic systems of competing species \cite{CLLL,ContiTerraciniVerziniNehari,CoTeVe05,DancerWangZhang,DancerWangZhangAddendum,NTTV1,RamosTavaresTerracini,STTZsurvey,SoaveZilio,SoaveZilio2,TavaresTerracini2}. These SMPs of bounded Euclidean domains have been completely characterized, both in terms of existence and in terms of regularity of the optimizing partitions. While an existence proof in the framework of quasi-open sets follows from the general theory in \cite{BucurButtazzoHenrot}, the existence of open partitions and optimal regularity of associated eigenfunctions was proved in \cite{CoTeVe05,HHT09}. Combining \cite{CaffarelliLin2,CoTeVe05,HHT09,RamosTavaresTerracini,TavaresTerracini}, one deduces the regularity of the inner free boundary $\Omega \cap \left(\cup_{i}\partial \omega_i\right)$ of any given optimal partition $(\omega_1,\ldots, \omega_k) $: up to a negligible singular set, it is a collection of $C^{1,\alpha}$ hypersurfaces. For the special case $p=1$, a more detailed description of the singular set \cite{Alper20,OgnibeneVelichkov2} or results on the interaction between  the inner free boundary and  $\partial \Omega \cap \left(\cup_{i}\partial \omega_i\right)$ (see \cite{OgnibeneVelichkov1})  are available.

\smallbreak

Here, our goal is to analyze, for the first time, such problems in a different context: \emph{unbounded} Euclidean domains, possibly of infinite volume, where in \eqref{eq:partition_bounded} we replace the first eigenvalue with the infimum of the spectrum of general classes of Schr\"odinger operators $-\Delta + V$ with Dirichlet boundary conditions.

Before going into details, we remark that the unbounded setting presents some fundamental differences, as the spectrum of the operators on the domain $\Omega$ can have a far more complicated nature, and so a number of new phenomena emerge; in particular, there might not even be a first Dirichlet eigenvalue. The actual existence of optimizers is not always true; even if optimizers exist, the cells of the optimal partitions may or may not actually have eigenvalues (ground states). We will present a threshold inequality for the energies related to the bottom of the essential spectrum of the operator on $\Omega$ (these quantities will be defined precisely below). Whenever that inequality is strict, optimal partitions exist, and the infimum of the spectrum of each cell is a simple isolated eigenvalue. As we will discuss below, this finds a parallel with the study of elliptic problems with critical nonlinearities.

In the process, we will introduce a natural weak formulation of the problem that involves functions rather than sets. However, unlike in the bounded domain case where the strong (i.e. \eqref{eq:partition_bounded}, with sets) and the weak formulations are equivalent, in our more general setting this is only necessarily true below the threshold level. Another remarkable difference in our setting occurs in the case $p=\infty$: it is no longer true in general that any optimal partition is an equipartition. 

We will illustrate and complement these general theorems with a number of examples in the last section. 

\smallbreak

Before continuing, we mention three papers which deal with shape optimization problems in unbounded domains, although of a different nature: the paper \cite{Buttazzo}, which deals with an optimization problem with an integral cost depending a potential $V$ varying in a certain class, with the goal of optimizing in $V$; papers \cite{Pacella2,Pacella_cylinder},which deal with 1-phase shape optimization problems related to overdetermined problems.

\subsection{Assumptions and notation}

In order to go into detail, we need to introduce some notation. As mentioned above, instead of considering just the Laplacian, we will assume there is a fixed underlying potential $V$ defined on $\Omega$ which will enter into the partition problem. From the theoretical point of view this makes no difference, but in practice allows for richer behavior and more examples. We will thus take the following assumption throughout the paper.

\begin{assumption}
\label{ass:basic}
Given $d\ge 2$, the set $\Omega \subset \R^d$ is open, and $V: \Omega \to \R$ is a fixed, measurable function such that $V\in L^\infty_{\text{loc}}(\Omega)$ and $V\ge 0$ a.e.\ in $\Omega$.
\end{assumption}
See Section~\ref{sec:omegavee} for a discussion of this assumption and, in particular, how positivity of $V$ can be weakened.

The Schr\"odinger operator $-\Delta + V$ (with Dirichlet boundary conditions) will thus take the place of the Dirichlet Laplacian. Now, since $\Omega$ might be unbounded, it can have subsets whose spectrum may no longer be discrete, but rather there may be essential spectrum (again, see below); in particular, in place of the first eigenvalue in \eqref{eq:partition_bounded}, we now consider the infimum of the spectrum of the Schr\"odinger operator on an arbitrary open set $\omega \subset \Omega$, which is characterized by
\begin{equation}
\label{eq:inf-spec}
    \infspec(\omega) := \inf_{u\in H^1_{0,V}(\omega)\setminus\{0\}} \frac{\int_\omega |\nabla u|^2 + V(x)|u|^2\, \mathrm dx}{\int_\omega |u|^2\, \mathrm dx},
\end{equation}
on the Sobolev space
\[
    H^1_{0,V}(\omega)=\left\{u\in H^1_0(\omega):\ \int_\omega V(x) |u|^2\, dx<\infty\right\}.
\]
Using a cutoff-function argument, the fact that $V\in L^\infty_{\text{loc}}$, and the definition of $H^1_0(\omega)$, it is not hard to check that $H^1_{0,V}(\omega)$ is the closure of $C^\infty_c(\Omega)$ for the norm $u\mapsto \|u\|_{H^1(\omega)}+\|\sqrt{V}u\|_{L^2(\omega)}$. Obviously $\infspec(\omega)$ depends on both the domain $\omega$ and the potential $V$; however, for simplicity of notation, and since we always consider the potential to be fixed, we will always write $\infspec (\omega)$ in place of, say, $\infspec (\omega, V|_\omega)$.

For any $k\in \N$ and $p\in [1,\infty]$, and a partition $\partition = (\omega_1, \ldots, \omega_k)$ into $k$ open, pairwise disjoint sets, or \emph{cells}, $\omega_i \subset \Omega$, we can then naturally define a corresponding energy functional $\energy{k}{p}$ as 
\begin{equation}
\label{eq:partition-functional}
    \energy{k}{p} (\partition) = \begin{cases}
        \left ( \sum_{i=1}^k \infspec(\omega_i)^p \right )^{1/p}, &\quad 1\le p < \infty, \\
    \max_{i=1, \ldots, k} \{ \infspec(\omega_i) \}, &\quad p =\infty.
    \end{cases}
\end{equation}
The problem of interest is thus the associated partition problem
\begin{equation}
\label{eq:partitionproblem}
\optenergy{k}{p}(\Omega) = \inf_{\partition= (\omega_1, \ldots, \omega_k)\in \partitionsetk} \energy{k}{p}(\partition),
\end{equation}
where the infimum is sought among all such admissible $k$-partitions of $\Omega$, the set of which we will denote by
\[
\partitionsetk=\{(\omega_1,\ldots, \omega_k):\ \omega_i \subset \Omega \text{ open, nonempty for all $i$},\  \omega_i\cap \omega_j=\emptyset \text{ for all $i\neq j$}\}.
\]
Observe that we do not impose connectedness of the partition elements. In the bounded case, connectedness is also not usually assumed, but the minimizers turn out to be connected anyway. For unbounded $\Omega$ it is possible to find minimizing partitions with disconnected cells;  in fact, unlike in the bounded case, it is also possible that minimizing partitions do not exhaust the whole domain. Examples of both phenomena can be found in Section \ref{sec:examples} below.

As mentioned above, and as in the case of bounded domains (see for instance \cite{CaffarelliLin2,CoTeVe05,RamosTavaresTerracini}), in order to study this partition problem we will study a corresponding \emph{relaxed problem}, where we work with $k$-tuples of functions defined on $\Omega$, rather than subsets of $\Omega$. In place of the infimum of the spectrum we introduce the Rayleigh quotient associated with the problem \eqref{eq:inf-spec}, that is, 
\begin{equation}
\label{eq:rayleigh-quotient}
\mathcal{R}_V(u):=\frac{\int_\Omega (|\nabla u|^2 + V(x)|u|^2)}{\int_\Omega |u|^2} \quad \text{ for } u\in H^1_{0,V}(\Omega)\setminus \{0\},
\end{equation}
so that, in particular, the infimum of the spectrum of $-\Delta+V$ on $\omega$ is given by
\[
\infspec (\omega) = \inf_{u \in H^1_{0,V}(\omega)\setminus \{0\}} \mathcal{R}_V(u).
\]
In this context, we will also need the associated quadratic form,
\begin{equation}
\label{eq:form}
a_V(u):=\int_\Omega (|\nabla u|^2+ V(x)|u|^2), \qquad u\in H^1_{0,V}(\Omega),
\end{equation}
where we may assume without loss of generality that all functions are real valued, so that $\mathcal{R}_V(u)=a_V(u)/\|u\|^2_2$ for any $u \not\equiv 0$.

With that background, we may introduce the relaxed version of the functional $\energy{k}{p}$, $k \geq 1$, $p \in [1,\infty]$, as
\begin{equation}
\label{eq:relaxed-functional}
\energyrel{k}{p}(u_1,\ldots, u_k)= \begin{cases}
        \left ( \sum_{i=1}^k \mathcal{R}_V(u_i)^p \right )^{1/p}, &\quad 1\le p < \infty, \\
    \max_{i=1, \ldots, k} \{ \mathcal{R}_V(u_i) \}, &\quad p =\infty, 
    \end{cases}
\end{equation}
and the relaxed problem corresponding to \eqref{eq:partitionproblem} then becomes
\begin{equation}
\label{eq:relaxed-partition-problem}
\optenergyrel{k}{p}(\Omega) = \inf_{\substack{u_1, \ldots, u_k\in H_{0,V}^1(\Omega)\setminus \{0\}\\ u_i \cdot u_j =0 \text{ a.e.},\ i \neq j}} \energyrel{k}{p}(u_1, \ldots, u_k).
\end{equation}

It is clear from the definitions and the assumption $V \ge 0$ that
\begin{equation}
\label{eq:inequality_relaxed}
0 \le \optenergyrel{k}{p}(\Omega)\le \optenergy{k}{p}(\Omega) < \infty.
\end{equation}
for all $k\in \N$ and $ p\in [1,\infty]$. It will follow from our main results that there is in fact always equality between the two infima. However, as we will see, $\optenergy{k}{p}(\Omega)$ may admit a minimizing partition when $\optenergyrel{k}{p}(\Omega)$ does not admit a minimizing $k$-tuple of functions.

Also relevant for our discussion of spectral minimal partitions will be the infimum of the essential spectrum of the corresponding Schr\"odinger operator on $\Omega$, which can be characterized via Persson's theorem (see \cite{HKS23,Pe60} or \cite[Theorem~2.6]{Si85}). Namely, the infimum of the essential spectrum of $-\Delta+V$ on $\omega\subseteq \Omega$ (or, for short, the infimum of the essential spectrum of $\omega$) is given by
\begin{align}
\label{eq:persson}
\Sigma(\omega)&=\sup_{\mathcal{K}\Subset \overline \omega}\lambda(\omega\setminus \mathcal{K})=\sup_{\mathcal{K}\Subset \overline \omega} 
\mathop{  \inf_{u\in H^1_{0,V}(\omega)\setminus\{0\}}}_{(\text{supp } u)\cap \mathcal{K}=\emptyset} \frac{\int_\omega (|\nabla u|^2 + V(x)|u|^2)}{\int_\omega |u|^2}
\end{align}
(see also \eqref{eq:from-persson-with-love} in Lemma~\ref{lem:domain-monotonicity}). We recommend \cite{HiSi96} as a general reference to the spectral theory of Schr\"odinger operators.

\begin{remark}
\label{rem:infess}
Note that any part of the spectrum below $\infess$ consists of isolated eigenvalues for which all the usual variational principles hold. What $\infess$ represents is the sharp energy level for the Rayleigh quotient below which \txtr{nonvanishing holds}: if for a sequence of $L^2$-normalized functions $u_n \in H^1_{0,V} (\Omega)$ one has $\sup_n \mathcal{R}_V(u_n) < \infess (\Omega)$, then $u_n\not\rightharpoonup 0$ in $H^1_{0,V}(\Omega)$ \txtr{(this is classical and will be shown in Step 4 of the proof of Theorem~\ref{thm:existence})}. Indeed, for this reason, this kind of quantity for more general functionals frequently appears in the calculus of variations literature; the link between the spectral-theoretic viewpoint and the calculus of variations viewpoint was explored in \cite{Ho19}. We point out that, when $\Omega$ is bounded or $V(x)\to \infty$ as $|x|\to \infty$, then $\Sigma (\Omega)=\infty$, see Remark \ref{rem:assumptions_V} for more details.
\end{remark}

\subsection{Statement of main results}

Our first result is as follows.

\begin{theorem}[Energy threshold]\label{thm:firstmain}
For any $k\in \N$ and any $p\in [1,\infty)$, we have
\begin{equation}
\label{eq:optimal-inequality-p}
\optenergy{k}{p}(\Omega) \le \min_{\ell=0,\ldots, k-1} \left(\optenergy{\ell}{p}(\Omega)^p + (k-\ell) \infess(\Omega)^p\right)^{1/p}=\left(\optenergy{k-1}{p}(\Omega)^p +  \infess(\Omega)^p \right)^{1/p}\le k^{1/p}\infess (\Omega)
\end{equation}
(where, by convention, for $k=0$ we set $\optenergy{0}{p}(\Omega) := 0$). For $p=\infty$, there exists a (not necessarily connected) $k$-partition $\partition$ such that
\begin{equation}
\label{eq:optimal-inequality-infty}
\optenergy{k}{\infty} (\Omega) \le \energy{k}{\infty}(\partition) \leq \infess(\Omega).
\end{equation}
\end{theorem}

 Observe that the second inequality in \eqref{eq:optimal-inequality-p} is, in general, strict. See Remark \ref{rem:bounds} below. On the other hand, it is straightforward to check that $\optenergy{k}{p}(\Omega)\geq k^{1/p}\lambda(\Omega)$.

We will refer to the quantities appearing in these bounds, that is,
\begin{equation}
\label{eq:threshold-values}
\threshold{k}{p} \equiv \threshold{k}{p}(\Omega)=
\begin{cases}
\left(\optenergy{k-1}{p}(\Omega)^p + \infess(\Omega)^p\right)^{1/p}   \qquad &\text{for } 1 \leq p<\infty,\\
\infess(\Omega) \qquad &\text{for } p=\infty,
\end{cases}
\end{equation}
as the threshold values.  We will also write $\thresholdrel{k}{p}$ for the corresponding weak or functional version, that is, the same quantity but where $\optenergy{k}{p}$ has been replaced by $\optenergyrel{k}{p}$. With this notation, the key inequality in Theorem~\ref{thm:firstmain} reads
\begin{equation}
\label{eq:threshold-inequality}
\optenergy{k}{p}(\Omega)\leq \threshold{k}{p}(\Omega)
\end{equation}
for all $k\in \N$ and all $p\in [1,\infty]$. Our second main result is the following complementary existence statement for the \emph{weak formulation} of the problem.

\begin{theorem}[Existence below the threshold]\label{thm:secondmain}
Given $k\ge 1$ and $1\le p \leq \infty$, if we have the strict inequality
\begin{equation}
\label{eq:strict-p}
\optenergyrel{k}{p}(\Omega) <  \thresholdrel{k}{p}(\Omega),
\end{equation}
then $\optenergyrel{k}{p} (\Omega)$ is achieved at some $u_1,\ldots, u_k\in H^{1}_{0,V}(\Omega)\setminus \{0\}$ with $u_i\cdot u_j\equiv 0$ for every $i\neq j$.
\end{theorem}

\begin{remark}
In the case $p=\infty$, the inequality \eqref{eq:optimal-inequality-infty} and a threshold condition analogous to \eqref{eq:strict-p} for the existence of SMPs were deduced in the sandbox case of unbounded \emph{quantum graphs} in \cite{HKS23}. There, however, the methods of proof were completely different (and the nature of the results was much more limited) due to the essentially one-dimensional nature of the problem considered there. Indeed, the variational (weak) formulation was not considered at all; rather, it was sufficient to study a suitable notion of convergence of sequences of subsets of the graph with respect to a Hausdorff distance, inspired by the framework introduced in \cite{KKLM21}. In higher dimensions, this metric fails completely to have the necessary compactness properties, and questions such as stability and regularity of sets also become an issue. Thus we need to revert to the relaxed variational model, as well as study the regularity of the supports. This also reveals, among other things, the new phenomenon, mentioned above, that the relaxed problem may not admit a solution even if the set formulation of the problem does.
\end{remark}

Inequality \eqref{eq:strict-p} is a compactness condition. Indeed, under these assumptions, for the proof of Theorem~\ref{thm:secondmain} we show that minimizing sequences strongly converge in $H^1_{0,V}(\Omega)$, see Theorem~\ref{thm:existence} below, cf.\ also Remark~\ref{rem:infess}. For $1 \leq p < \infty$, this condition recalls analogous inequalities and conditions for elliptic problems with critical nonlinearities. For instance, in the celebrated Brezis--Nirenberg \cite{BrezisNirenberg} problem 
\[
-\Delta u+\lambda u=|u|^{2^*-2}u \text{ in } \Omega, \quad u=0 \text{ on } \partial \Omega
\]
(with $\Omega$ bounded), a compactness condition is related to the energy being strictly smaller than the least energy level $c_{\R^d}$ of the equation $-\Delta U=U^{2^*-1}$ in $\R^d$.  For variational semilinear elliptic systems  with competition terms and critical exponents, a compactness condition that ensures the existence of least energy solutions with all components non zero is guaranteed below a threshold level, in the same spirit as \eqref{eq:strict-p}; for this, we refer for instance to \cite[Proposition 3.3 and Lemma 4.2]{ClappPistoiaTavares}, \cite[Lemma 4.10]{ClappSzulkin2019} and \cite[Lemma 5.1]{ChenZouARMA2012}, see also \cite{ChenZou2,TavaresYou,TavaresYouZou}. Finally, for results of a similar flavor in the framework of semilinear equations on \emph{noncompact} metric graphs, see \cite[Theorem~1.3]{deCosteretal2023}.

\smallbreak

One of the advantages of dealing with the weak formulation $\optenergyrel{k}{p}$ is that it allows us to adapt previously known regularity techniques. In particular, this will allow us to prove that, actually, $\optenergyrel{k}{p}(\Omega)$ and $\optenergy{k}{p}(\Omega)$ always coincide (and, thus, also $\threshold{k}{p}(\Omega)$ and $\thresholdrel{k}{p}(\Omega)$).

\begin{theorem}[Regularity of the minimizers]
\label{thm:regularity}
Take $k \in \N$. Let $p\in [1,\infty)$ and assume there is a minimizer $(u_1,\ldots, u_k)$ of $\optenergyrel{k}{p}(\Omega)$. Then the following statements hold.
\begin{itemize}
    \item[(1)] For all $i=1,\ldots,k$, we have $u_i\in C^{0,1}(\Omega)$. In particular, $(\omega_1,\ldots, \omega_k):=(\{u_1\neq 0\},\ldots, \{u_k\neq 0\})\in \partitionset{k}$ is a minimizing partition for $\optenergy{k}{p} (\Omega)$, and
    \[
    -\Delta u_i+V(x)u_i=\lambda(\omega_i)u_i \quad \text{ in } \omega_i.
    \]
    \item[(2)] Suppose in addition that $V \in C^1(\Omega)$. Then the nodal set $\Gamma:=\{x\in \Omega: u_i(x)=0 \text{ for every i}\}$ can be decomposed as a union of a regular and a singular part, $\Gamma=\mathscr{R} \cup \mathscr{S}$, where $\mathscr{R} \cap \mathscr{S}=\emptyset, \mathscr{R}$ is a collection of $(d-1)$-dimensional $\mathcal{C}^{1, \alpha}$-surfaces, and $\mathscr{S}$ is a closed subset of $\Gamma$ with Hausdorff measure less than or equal to $d-2$. In particular, the partition exhausts the whole domain, in the sense that 
    \[
    \overline \Omega=\cup_{i=1}^k \overline{\omega}_i, \text{ and } \Gamma=\cup_{i=1}^k (\partial \omega_i\cap \Omega).
    \]
    Moreover,
    \begin{itemize}
\item[(a)]  given $x_0 \in \mathscr{R}$, there exist $i \neq j$ such that
$$
\lim _{x \rightarrow x_0^{+}}\left|\nabla u_i(x)\right|=\lim _{x \rightarrow x_0^{-}}\left|\nabla u_j(x)\right| \neq 0,
$$
where $x \rightarrow x_0^{ \pm}$are the limits taken from opposite sides of $\mathscr{R}$;
\item[(b)] for $x_0 \in \mathscr{S}$, we have
$$
\lim _{x \rightarrow x_0}\left|\nabla u_i(p)\right|=0 \quad \text { for every } i=1, \ldots, k;
$$

\item[(c)] in dimension $d=2$,  $\Gamma$ locally consists  of a finite collection of regular curves meeting with equal angles at isolated singular points. 
    \end{itemize}
\end{itemize}

On the other hand, if $p=\infty$ and the threshold condition \eqref{eq:strict-p} is satisfied, then there \emph{exists} a minimizer $(u_1,\ldots, u_k)$ of $\optenergyrel{k}{\infty} (\Omega)$ for which (1)-(2) above are true.
\end{theorem}

A direct consequence of Theorem \ref{thm:regularity}-(1) is that  $\optenergy{k}{p}(\Omega) = \optenergyrel{k}{p}(\Omega)$ whenever the threshold condition \eqref{eq:strict-p} is satisfied; it turns out that this identity is true independently of \eqref{eq:strict-p}. Moreover, for $p<\infty$, if $\optenergyrel{k}{p}(\Omega)$ is attained, then there is an optimal partition attaining $\optenergy{k}{p}(\Omega)$.  A natural question is then what happens in the remaining cases. Remarkably, it turns out that the answer depends on $p$. In summary:

\begin{proposition}
\label{prop:weak-and-strong}
For any $k \in \N$ and any $ p\in [1,\infty]$, we have
\begin{equation}
\label{eq:weak-and-strong}
    \optenergy{k}{p}(\Omega) = \optenergyrel{k}{p}(\Omega)
    \quad \text{and} \quad \threshold{k}{p}(\Omega) = \thresholdrel{k}{p}(\Omega).
\end{equation}
In particular, if $\optenergy{k}{p}(\Omega) < \threshold{k}{p}(\Omega)$, then $\optenergy{k}{p}(\Omega)$ is achieved at some $k$-partition $(\omega_1,\ldots,\omega_k)$ consisting of connected, open sets. Moreover: 
\begin{itemize}
    \item[(1)] When $p=\infty$, there always exists a partition attaining $\optenergy{k}{\infty}(\Omega)$;
    \item[(2)] There exist a domain $\Omega$ and a potential $V$ such that, for any $1 \leq p \leq \infty$, $\optenergy{k}{p}(\Omega)$ is attained but $\optenergyrel{k}{p}(\Omega)$ is not attained;
    \item[(3)] There exist a domain $\Omega$ and a potential $V$ such that, for any $1\leq p<\infty$, neither $\optenergy{k}{p}(\Omega)$ nor $\optenergyrel{k}{p}(\Omega)$ is attained.
    \end{itemize}
\end{proposition}

The examples mentioned in Proposition \ref{prop:weak-and-strong}-(2),(3) correspond to Examples \ref{ex:strip} and \ref{ex:fortschrittchen}, respectively.

\medbreak

\begin{remark}\label{oldremark}
Fix $1 \leq p \leq \infty$ and suppose the threshold condition \eqref{eq:strict-p} is satisfied. We claim that any minimizing partition $(\omega_1,\ldots, \omega_k)$ of $\optenergy{k}{p} (\Omega) = \optenergyrel{k}{p} (\Omega)$ satisfies
    \begin{equation}
    \label{eq:all-below}
    \infspec(\omega_i)<\infess(\Omega)\le \infess(\omega_i) \qquad \text{ for $i=1,\ldots, k$.}
    \end{equation}
Indeed, while for  $p=\infty$ this follows directly from the definition of $\optenergy{k}{\infty}$, for $1 \le p < \infty$ condition \eqref{eq:strict-p} implies that, for any fixed $j=1,\ldots, k$,
    \begin{displaymath}
        \sum_{i=1}^k \infspec(\omega_i)^p = \optenergy{k}{p}(\Omega)^p <  \optenergy{k-1}{p}(\Omega)^p + \infess(\Omega)^p \le \sum_{\substack{i=1\\ i\neq j}}^k \infspec(\omega_i)^p + \infess(\Omega)^p,
    \end{displaymath}
whence \eqref{eq:all-below}.
   
In particular, under the threshold condition \eqref{eq:strict-p}, any partition element $\omega_i$ of any minimizing partition of $\optenergy{k}{p}(\Omega)$ admits a ground state $u_i$ (that is, an eigenfunction), $\lambda(\omega_i)$ is a discrete eigenvalue and $(u_1,\ldots, u_k)$ is a minimizer for $\optenergyrel{k}{p}$. Combined with Theorem \ref{thm:regularity}-(1), this shows that, for  $p<\infty$ and under \eqref{eq:strict-p}, there is a one-to-one correspondence between minimizers of $\optenergyrel{k}{p}(\Omega)$ and $\optenergy{k}{p}(\Omega)$. Proposition~\ref{prop:weak-and-strong} shows that this is not, in general, the situation.
\end{remark}

Proposition~\ref{prop:weak-and-strong}-(2) showed that, for $p < \infty$, the strong formulation of the problem, $\optenergy{k}{p}$, may have a solution, while the weak formulation, $\optenergyrel{k}{p}$, does not; this is related to the non-existence of ground states of the partition cells $\omega_i$. For the case $p=\infty$, there is a slightly different example of how the behavior changes when the threshold condition fails.

\begin{proposition}[Equipartitions]\label{prop:egalite}
Let $p=\infty$. 
    \begin{itemize}
    \item[(1)] Take $k\geq 1$ and suppose that the threshold condition \eqref{eq:strict-p} holds. Then there exists a  minimizing $k$-partition $\partition = (\Omega_1,\ldots,\Omega_k)$ for $\optenergy{k}{\infty}(\Omega)$ for which
    \begin{displaymath}
        \lambda_1 (\Omega_1) = \ldots = \lambda_1 (\Omega_k) = \optenergy{k}{\infty} (\Omega),
    \end{displaymath}
    that is, $\partition$ is an equipartition.
    \item[(2)] There exist a domain $\Omega$, a potential $V$ and $k \geq 1$ such that \eqref{eq:strict-p} does not hold and there exists a minimizing $k$-partition for $\optenergy{k}{\infty}(\Omega)$ which is \emph{not} an equipartition.
    \end{itemize}
\end{proposition}

We strongly expect that the conclusion of (1) should hold for \emph{any} minimizing $k$-partition of $\Omega$ if the threshold condition \eqref{eq:strict-p} holds (for $p=\infty$), but it would take us too far afield to include a proof here. The particular solution mentioned in (1) is related to the statement regarding $p=\infty$ in Theorem \ref{thm:regularity} (which we also expect to be true for any optimal partition); regarding (2), see Example \ref{ex:liberte}.

\subsection{Discussion and complementary results.}
In this subsection we make several comments related to our previous results. First, we discuss weakening or changing some of our assumptions, in particular allowing $V$ to have a negative part; we will also briefly consider the special case where $-\Delta+V$ has compact resolvent. Second, we return to the two bounds which appear in the threshold inequality \eqref{eq:optimal-inequality-p} in Theorem~\ref{thm:firstmain}. We finish by analyzing the behavior of $\optenergy{k}{p}(\Omega)$ and the optimal threshold $\threshold{k}{p}(\Omega)$ as functions of $k$ and $p$.

\subsubsection{On the assumptions on $\Omega$ and $V$}
\label{sec:omegavee}

\begin{remark}\label{rem:assumptions_V}
If $\Omega$ has finite volume or $V(x) \to \infty$ as $|x|\to\infty$, then the embedding $H^1_{0,V}(\Omega)\hookrightarrow L^2(\Omega)$ is compact. Thus the corresponding Schr\"odinger operator has compact resolvent on $\Omega$ and, in particular, $\infess(\Omega) = \infty$ and $\thresholdrel{k}{p}(\Omega)=\infty$ for every $p\in [1,\infty]$. In particular, the compactness condition \eqref{eq:strict-p}, i.e.
\[
 \optenergy{k}{p}(\Omega) = \optenergyrel{k}{p}(\Omega) < \thresholdrel{k}{p}(\Omega) = \threshold{k}{p}(\Omega),
\]
is always satisfied, and the conclusions of Theorems~\ref{thm:secondmain} and~\ref{thm:regularity} hold.

This shows in particular that our results include and generalize \cite{CoTeVe05} and \cite[Section 8.2, p.~315]{TavaresTerracini}, which deal with the case of a bounded domain $\Omega$. For examples of unbounded domains $\Omega$ for which \eqref{eq:strict-p} holds for some $k \geq 1$, see Examples~\ref{ex:compact-resolvent} and~\ref{ex:threshold-yes}.
\end{remark}

\begin{remark}
\label{rem:negative-potentials}
    If $p=1$ or $p=\infty$, the assumption that $V \ge 0$ can be weakened to $V_- := \max \{-V,0\} \in L^\infty(\Omega)$, as adding a constant $c \in \R$ to the equation, i.e. considering
    \[
        -\Delta u + (V+c)u = (\lambda+c) u,
    \]
    will only shift $\infspec$ and $\infess$ by the same constant $c$. 
    
    On the other hand, for $1<p<\infty$, the terms appearing in \eqref{eq:optimal-inequality-p} may no longer be well defined, since it is possible that $\lambda(\omega)$ and $\Sigma(\Omega)$ are negative. In this situation, the optimization problem $\optenergy{k}{p}(\Omega)$ may no longer make sense and one would have to consider other functionals such as 
    \[
    (\omega_1,\ldots, \omega_k)\mapsto \sum_{i=1}^k |\lambda(\omega_i)|^{p-1}\lambda(\omega_i).
    \]
    We will not consider these problems here.
\end{remark}

\subsubsection{On the threshold inequality \eqref{eq:optimal-inequality-p}}

\begin{remark}
\label{rem:bounds}
In this remark, we will explore a little further the nature of the two inequalities appearing in \eqref{eq:optimal-inequality-p}. For $1\leq p < \infty$, we show that both $\threshold{k}{p}(\Omega)<k^{1/p}\infess (\Omega)$ and $\threshold{k}{p}(\Omega)=k^{1/p}\infess (\Omega)$ are possible in specific cases, relating this to the relationship between the bottom of the spectrum $\infspec(\Omega)$ and the bottom of the essential spectrum $\infess(\Omega)$. In particular, we remark how condition $\threshold{k}{p}(\Omega)<k^{1/p}\infess (\Omega)$, unlike the corresponding condition for $p=\infty$  (namely $\threshold{k}{\infty}(\Omega)<\infess (\Omega)$), does not automatically imply the existence of a minimal $k$-tuple for neither $\optenergyrel{k}{p}(\Omega)$ nor $\optenergy{k}{p}(\Omega)$.
\begin{enumerate}
    \item Assume $\infspec (\Omega) < \infess (\Omega)$, which implies that $\lambda(\Omega)$ is an isolated simple eigenvalue.  It follows from the fact that $\optenergy{1}{p}(\Omega) = \optenergyrel{1}{p} (\Omega) = \infspec (\Omega)$ for any $p$ and inequality \eqref{eq:optimal-inequality-p}  for $\ell=1$ that
    \begin{equation}\label{eq:rem_strictineq}
        \optenergy{k}{p}(\Omega) = \optenergyrel{k}{p} (\Omega) < k^{1/p}\infess (\Omega)
    \end{equation}
    for all $k \ge 1$ and all $1 \le p < \infty$. In Example~\ref{ex:fortschrittchen}, we show a situation where
    \[
    \lambda(\Omega)<\Sigma(\Omega)\quad \text{ and }\quad \optenergy{k}{p} (\Omega) =\threshold{k}{p}(\Omega),
    \]
    and neither $\optenergy{k}{p}(\Omega)$ nor $\optenergyrel{k}{p}(\Omega)$ admit a minimizer for any $1 \leq p < \infty$ (concretely for $k=2$). This shows that, for some domains and some potentials, it may happen that
    \[
    \threshold{k}{p}(\Omega)<k^{1/p}\infess (\Omega) \quad \text{ and } \quad \optenergy{k}{p}(\Omega),\ \optenergyrel{k}{p}(\Omega) \text{ are not attained.}
    \]
    In particular, condition \eqref{eq:rem_strictineq} does not imply the existence of a minimal $k$-tuple for $\optenergyrel{k}{p}(\Omega)$.
    \item Now, assume $\infspec (\Omega) = \infess(\Omega)$, in which case there may or may not still exist a ground state at the bottom of the spectrum, see, e.g., \cite[Section~3.4]{BeSh91} for $d \geq 5$. Then, recalling that $\optenergy{k}{p}(\Omega)\geq k^{1/p}\lambda(\Omega)$, one obtains immediately from \eqref{eq:optimal-inequality-p} that
    \[
    \optenergy{k}{p} (\Omega) = \threshold{k}{p} (\Omega) = k^{1/p} \infess(\Omega)
    \]
    for $1\leq p<\infty$.
    
    Actually, the partition $\partition$ satisfying \eqref{eq:optimal-inequality-infty} in Theorem~\ref{thm:firstmain}, that is, with $\energy{k}{\infty}(\partition) \le \infess(\Omega)$, can be used as a test partition; this yields
    \begin{equation}
        k^{1/p} \infess(\Omega) = \optenergy{k}{p}(\Omega) \le \energy{k}{p}(\partition) \le k^{1/p} \infess(\Omega).
    \end{equation}
    Hence, in this case, there is always a partition attaining $\optenergy{k}{p} (\Omega)$, although $\optenergyrel{k}{p}(\Omega)$ may not be attained.
\end{enumerate}
\end{remark}

\subsubsection{Monotonicity of the energy levels with respect to $p$ and $k$}

Next, we consider the issue of the continuity and monotonicity of the studied quantities in $p$ and $k$. The former is actually a key factor in the proof of Theorem \ref{thm:regularity} for $p=\infty$.

\begin{proposition}[Behavior in $p$]
\label{prop:p}
Let $k \ge 1$. Then:
\begin{itemize}
\item[(1)] The functions
\begin{equation*}
    p \mapsto \optenergy{k}{p}(\Omega)=\optenergyrel{k}{p} (\Omega) \quad \text{ and }\quad p \mapsto \threshold{k}{p}(\Omega)  = \thresholdrel{k}{p}(\Omega)
\end{equation*}
are continuous and non-increasing in $p \in [1,\infty]$.
\end{itemize}
Now let $p \in [1,\infty]$ and suppose the threshold condition \eqref{eq:strict-p} is satisfied for this $p$ (in particular, by Theorem \ref{thm:secondmain}, a minimizer exists for $\optenergyrel{k}{p}(\Omega)$). Then:
\begin{itemize}
\item[(2)] There exists a neighborhood of $p \in [1,\infty]$ in which \eqref{eq:strict-p} holds. 
\item[(3)] (Continuity of the minimizers.) Suppose $p_n \to p \in [1,\infty]$ and, for each $n$, let $U_{p_n} = (u_{1,n},\ldots,u_{k,n}) \in H^1_{0,V} (\Omega)^k$ be a minimizing $k$-tuple for $\optenergyrel{k}{p_n} (\Omega)$, normalized so that $\|u_{i,n}\|_2 = 1$ for all $i=1,\ldots,k$. Then, up to a subsequence, $U_{p_n}$ converges in $H^1_{0,V}(\Omega)^k$ to a minimizing $k$-tuple $U_p$ for $\optenergyrel{k}{p}(\Omega)$.
\end{itemize}
\end{proposition}

\begin{proposition}[Monotonicity in $k$]\label{prop:aok}
Take $k\in \N$. Then 
\[
\optenergyrel{k}{p} (\Omega) \leq \optenergyrel{k+1}{p} (\Omega) \text{ for all $p\in [1,\infty]$},\qquad  
\thresholdrel{k}{p} (\Omega) \leq \thresholdrel{k+1}{p} (\Omega) \text{ for all $p\in [1,\infty)$.} 
\]
For $p \in [1,\infty)$, the inequalities are strict if the threshold condition \eqref{eq:strict-p} is satisfied for $\optenergyrel{k}{p} (\Omega)$.
\end{proposition}

Note that if, for some $k_0 \geq 1$, $\optenergyrel{k_0}{\infty} (\Omega)$ is equal to the threshold value, then it is immediate from Theorem~\ref{thm:firstmain} that the functions $k \mapsto \optenergyrel{k}{\infty} (\Omega) =  \thresholdrel{k}{\infty} (\Omega) =\infess(\Omega)$ are constant for $k \geq k_0$.

In the next remark, using the monotonicity in $p$, we explore a simple inequality that shows how the potential $V$ bounds  the number of $k$'s for which $\optenergy{k}{p}(\Omega)$ is below a certain value, 

\begin{remark}\label{rmk:clrestimate} 
    As in the case of bounded domains, it is immediate from the min-max principle for the $k$-th eigenvalue $\lambda_k (\Omega)$, if such an eigenvalue exists below $\infess(\Omega)$, that
    \begin{displaymath}
        \optenergy{k}{\infty}(\Omega) \geq \lambda_k (\Omega).
    \end{displaymath}
    Indeed, if $\optenergy{k}{\infty} (\Omega) = \infess (\Omega)$ there is nothing to prove; if $\optenergy{k}{\infty} (\Omega) < \infess (\Omega)$ and $(u_1,\ldots,u_k)$ is a minimizing family for $\optenergyrel{k}{\infty}(\Omega)$, then one has a family of $k$ orthogonal test functions for $\lambda_k(\Omega)$. It follows from the monotonicity of $p \mapsto \optenergy{k}{p}(\Omega)$ that in fact
    \begin{equation}
    \label{eq:energy-eigenvalue}
        \optenergy{k}{p}(\Omega)\ge \optenergy{k}{\infty}(\Omega) \ge \lambda_k(\Omega)
    \end{equation}
    for all $1 \leq p \leq \infty$, as long as $\lambda_k(\Omega) < \infess (\Omega)$ exists. 
    
    This can be immediately rephrased in terms of eigenvalue counting functions; we will do this as it will be useful in some of the examples, in particular Examples~\ref{ex:strip-ball} and~\ref{ex:fortschrittchen}. Given $c<\Sigma(\Omega)$, let
$N(c, -\Delta+V)$ be the number of  eigenvalues of $-\Delta+V$ less than or equal to $c$, and, for $p\in [1,\infty]$, let 
\[
\widetilde N_p(c,-\Delta +V):=\sharp\{k: \optenergyrel{k}{p}\leq c\}=\max\{k:\ \optenergyrel{k}{p}(\Omega)\leq c \},
\]  
so that \eqref{eq:energy-eigenvalue} may be rephrased as
\begin{equation}
\label{eq:counting-bounds}
    \widetilde N_p(c,-\Delta +V) \leq \widetilde N_\infty(c,-\Delta +V) \leq N(c, -\Delta+V).
\end{equation}
In particular, the number of $k$ for which $\optenergyrel{k}{p} < \infess(\Omega)$ is no larger than the number of eigenvalues (counted with multiplicities) of $-\Delta+V$ below $\infess(\Omega)$.
Note that \eqref{eq:counting-bounds} immediately implies various estimates on the number of such spectral minimal partitions. For example, using the well known CLR-inequality     \cite{CLR1,CLR2,CLR3} (see also \cite[Theorem 4.31]{FLW23}) which bounds the number of negative Schr\"odinger eigenvalues, for $d\ge 3$ and $(V-c)_-:= \max\{-V+c, 0\}\in L^{d/2}(\Omega)$, we have
\begin{equation}\label{eq:CLR}
    \widetilde N_p(c,-\Delta +V)\leq N(c, -\Delta + V) \le L_{0,d} \int_{\Omega} |(V-c)_-|^{d/2}\, \mathrm dx
\end{equation}
for some constant $L_{0,d}>0$.
\end{remark}

\subsection{Description of examples}
\label{sec:description-examples}
Since many of the phenomena we are studying are new, we will give a fairly large number of examples illustrating what can happen, to give a more complete picture of the possible behavior regarding the existence and nonexistence in scenarios beyond the existence theory in Theorem~\ref{thm:secondmain}. For ease of reference, we will now provide a complete list of the examples, which will all be given in Section~\ref{sec:examples}, together with a short description of what they show:
\begin{itemize}
    \item In Example~\ref{ex:compact-resolvent}, for any $p\in [1,\infty]$, we provide an example where minimizers exist due to the compact resolvent of the underlying Schrödinger operator, even though the domain is unbounded, thus generalizing the ``classical'' existence results for bounded domains (recall Remark~\ref{rem:assumptions_V}).
    \item In Example~\ref{ex:threshold-yes}, we give examples where the operator does not have compact resolvent but the threshold condition holds for any given $p\in [1,\infty]$ (and thus there exist solutions of both the relaxed and the strong formulations).
     \item In Example~\ref{ex:nopotential}, for any $p\in [1,\infty]$, we show that, for $1\leq p\leq \infty$, even when the threshold condition does not hold and $\optenergyrel{k}{p}$ is not attained, minimizing partitions (for the strong formulation $\optenergy{k}{p}$) may still exist. Note that in this case the cells of the minimizing partitions need not be connected.
    \item In Example~\ref{ex:strip}, we give a more sophisticated example of the previous phenomenon, where the cells of the minimizing partition \emph{are} connected, although only for $p=\infty$ and $d=2$. We also show in Remark~\ref{rem:super-strip} that the same conclusion holds for \emph{any} $\Omega$ in dimension $d \geq 3$, at least when $p=\infty$.
    \item In Example~\ref{ex:liberte}, we construct a minimizing partition when $p=\infty$ which is not an equipartition, in a case where the strict threshold condition \eqref{eq:strict-p} fails; this corresponds to Proposition \ref{prop:egalite}-(2).  
    \item In Example~\ref{ex:strip-ball}, we give an example where for some  the threshold condition~\eqref{eq:strict-p} fails, and there exists a spectral minimal partition where some cells admit ground states, but others do not.
    \item Finally, in Example~\ref{ex:fortschrittchen} (based on a domain and potential constructed in the auxiliary Example~\ref{ex:durchbruechchen}) we give a domain $\Omega$ and a potential where there are finitely many eigenvalues below the essential spectrum; for $k \geq 2$, the threshold condition \eqref{eq:strict-p} does not hold for any $1 \leq p \leq \infty$, neither $\optenergy{k}{p}(\Omega)$ nor $\optenergyrel{k}{p}(\Omega)$ admits a minimizer for $p < \infty$, but $\optenergyrel{k}{\infty}(\Omega)$ may or may not admit a minimizer, in dependence on whether there is a $k$-th eigenvalue $\lambda_k (\Omega)$ equal to $\infess (\Omega)$.  This shows Proposition \ref{prop:weak-and-strong}-(3). See also Remark~\ref{rmk:clrestimate}.
\end{itemize}

\subsection{Structure of the paper and strategy of the main proofs.}
Let us briefly describe the structure of the paper. We will commence with the proofs of Theorem~\ref{thm:firstmain} and Theorem~\ref{thm:secondmain} in Sections~\ref{sec:proof-firstmain} and~\ref{sec:existence}, respectively.

 The proof of Theorem~\ref{thm:firstmain}, establishing the energy threshold (upper bound) for the optimal energy $\optenergy{k}{p}(\Omega)$, is relatively direct, based on finding a suitable test partition. One uses the Persson characterization \eqref{eq:persson} of $\infess$ together with monotonicity and continuity results for $\infspec$ to find large concentric ring-type subsets $\omega_1,\ldots,\omega_k$ of $\Omega$, each of which has its infimum of the spectrum bounded from above by approximately $\infess (\Omega)$ (see \eqref{eq:olympic-rings} and~\eqref{eq:lordoftherings}).

The proof of existence of a solution of the relaxed problem in Section~\ref{sec:existence}, see Theorem~\ref{thm:existence}, which covers Theorem~\ref{thm:secondmain}, is far more delicate. The key idea is that we recover a notion of compactness for sequences of functions in $H^1_{0,V}(\Omega)$ with Rayleigh quotients below $\infess (\Omega)$. To exploit this, we still need a careful diagonal-type cut-off argument, or decomposition, see Lemma~\ref{lemma:existence_of_cutoff_functions}, which we then combine with the so-called IMS localization formula (see \cite[Section~2]{Si82}) to split each $k$-tuple of functions in the minimizing sequence for the relaxed problem into two parts, a part escaping to infinity and a part supported in an expanding ball around the origin, in such a way that the Rayleigh quotient of each part is manageable. We then need a careful analysis of the two parts to show that the part expanding out from the origin has a nontrivial limit; for each $\ell=0,\ldots, k-1$, the condition
\[
\optenergy{k}{p}(\Omega) < \left(\optenergy{\ell}{p}(\Omega)^p + (k-\ell) \infess(\Omega)^p\right)^{1/p}
\]
appearing in \eqref{eq:strict-p} prevents the limit from having $k-\ell$ trivial components. One can then show directly that this limit is indeed a minimizing $k$-tuple for the relaxed problem.

We then prove \txtr{a weaker version of} Proposition~\ref{prop:p} in Section~\ref{sec:continuity}, see Proposition~\ref{prop:p-continuity}, as it delivers us a continuity result needed in Theorem~\ref{thm:regularity} for the case $p=\infty$. This proof follows from rather standard ideas about comparing $p$-norms.

We then finally proceed with the proof of Theorem~\ref{thm:regularity} in Section~\ref{sec:regularity}, in which we also prove the remaining propositions from Section~\ref{sec:intro}. The proof of the former is a direct adaptation of \cite{CoTeVe05, HHT09}, where the domain $\Omega$ is bounded and $V\equiv 0$; we simply point out some differences to our situation. We also prove Propositions \ref{prop:weak-and-strong}, \ref{prop:egalite} and \ref{prop:aok} in this section, \txtr{where we also prove Proposition~\ref{prop:p} in its full form.}

As mentioned, the examples will be given, in the order indicated above, in the final Section~\ref{sec:examples}.

\section{Proof of Theorem~\ref{thm:firstmain}}
\label{sec:proof-firstmain}

We start with a couple of auxiliary results on the behavior of $\infspec (\omega)$ as a function of the domain $\omega \subset \Omega$.  Here and throughout we use $\subset$ to mean inclusion, which need not be strict.

In the sequel, we will use the notation
\begin{displaymath}
\begin{aligned}
    B_r \equiv B_r(0) &:= \{ x \in \R^d: |x| < r\},\\
    A_{r,R} \equiv A_{r,R}(0) &:= \{ x \in \R^d: r < |x| < R\}
\end{aligned}
\end{displaymath}
for open balls and annuli centered at $0$, respectively.

We will also use the alternative notation
\[
    \mathcal{R}_{V,\omega} (u) := \mathcal{R}_{V} (u)
\]
for the Rayleigh quotient of a function $u \in H^1_{0,V} (\omega) \setminus \{0\} \hookrightarrow H^1_{0,V} (\Omega)$, if we wish to emphasize the domain on which $u$ is defined.

We start with a basic domain monotonicity and convergence result.

\begin{lemma}\label{lem:domain-monotonicity}
    Under our standing assumptions on $\Omega$ and $V$, let $\omega_1 \subset \omega_2 \subset \Omega$ be any open sets, not necessarily connected. Then
    \begin{equation}
    \label{eq:domain-monotonicity}
    \infspec(\omega_1) \geq \infspec(\omega_2).
    \end{equation} 
    Moreover, for any $\omega \subset \Omega$ open, we have
    \begin{equation}
    \label{eq:to-infinity-and-beyond}
         \infspec (\omega \cap B_r) \searrow \infspec (\omega)  \quad \text{ as }r\to \infty 
    \end{equation}
    and
    \begin{equation}
    \label{eq:from-persson-with-love}
       \infspec (\omega \setminus \overline{B_r}) \nearrow \Sigma (\omega)  \quad \text{ as }r\to \infty. 
    \end{equation}
\end{lemma}

\begin{proof}
    The inequality \eqref{eq:domain-monotonicity} is a direct consequence of the variational characterization \eqref{eq:inf-spec} of $\infspec$, together with the natural inclusion $H^1_{0,V}(\omega_1) \subset H^1_{0,V}(\omega_2)$.

    For \eqref{eq:to-infinity-and-beyond}, we note that the inequality $\infspec (\omega \cap B_r) \ge \infspec (\omega)$ for all $r>0$ is an immediate consequence of \eqref{eq:domain-monotonicity}. The other inequality follows from the variational characterization and the density (by construction) of $C_c^\infty (\omega)$ in $H^1_{0,V}(\omega)$: in particular, in \eqref{eq:inf-spec}, we may equally infimize over $C_c^\infty (\omega)$ in place of $H^1_{0,V}(\omega)$. But now, any $u \in C_c^\infty (\omega)$ belongs to $H^1_{0,V}(\omega \cap B_r)$ for all $r=r(u)>0$ large enough. The claim follows.

    Finally, for \eqref{eq:from-persson-with-love}, since $\overline{B_r}$ is compact, the characterization \eqref{eq:persson} directly implies that
    \begin{displaymath}
        \infspec (\omega \setminus \overline{B_r}) \le \infess({\omega})
    \end{displaymath}
    for all $r>0$. Now let $\mathcal{K}_n$ be any maximizing sequence of domains in \eqref{eq:persson}. For each $n \in \N$, there exists $r_n>0$ such that $\mathcal{K}_n \subset \overline{B_{r_n}}$; hence, by \eqref{eq:domain-monotonicity},
    \begin{displaymath}
        \infspec (\omega \setminus \overline{B_{r_n}}) \ge \infspec (\omega \setminus \mathcal{K}_n) \to \infess (\omega).
    \end{displaymath}
    Since the function $r \mapsto \infspec (\omega \setminus \overline{B_r})$ is monotonically increasing in $r$ by \eqref{eq:domain-monotonicity}, the claim follows.
\end{proof}

\begin{lemma}
\label{lem:annulus-limit}
    Let $\omega \subset \Omega$ be any open subset of $\Omega$. Then, for any $r>0$,
    \begin{displaymath}
        \lim_{R \to \infty} \infspec (\omega \cap A_{r,R}) = \infspec (\omega \setminus \overline{B_r}) \leq \infess (\omega).
    \end{displaymath}
    In particular, for any $r>0$ and any $\varepsilon > 0$ there exists $R=R(r,\varepsilon) > r$ such that
    \begin{equation}\label{eq:annulus-control}
        \infspec (\omega \cap A_{r,R}) < \infess (\omega) + \varepsilon.
    \end{equation}
\end{lemma}

\begin{proof}
     By Lemma~\ref{lem:domain-monotonicity}, we immediately obtain that $R \mapsto \infspec (\omega \cap A_{r,R})$ is a decreasing function bounded from below by $\infspec (\omega \setminus \overline{B_r})$. 
    
    On the other hand, by density of $C_c^\infty$ is dense in $H^1_{0,V}$, there exist $\varphi_n \in C_c^\infty (\omega \setminus \overline{B_r})$ such that $\mathcal{R}_{V,\omega \setminus \overline{B_r}} (\varphi_n) \to \infspec (\omega \setminus \overline{B_r})$; since, for every $n\in\N$, there exists $R_n > 0$ such that $\varphi_n \in C_c^\infty (\omega \cap A_{r,R_n})$, it follows immediately that
    \begin{displaymath}
        \infspec (\omega \cap A_{r,R_n}) \leq \mathcal{R}_{V,\omega \cap A_{r,R_n}}(\varphi_n) = \mathcal{R}_{V, \omega \setminus \overline{B_r}}(\varphi_n) \to \infspec (\omega \setminus \overline{B_r}),
    \end{displaymath}
    whence there is equality in the limit.

    The inequality $\infspec (\omega \setminus \overline{B_r}) \leq \infess (\omega)$ is a direct consequence of \eqref{eq:from-persson-with-love}. The inequality \eqref{eq:annulus-control} now follows immediately from the other parts of the lemma.
\end{proof}

\begin{proof}[Proof of Theorem \ref{thm:firstmain}]
We keep the notation from above. We also assume that $\Sigma(\Omega) < \infty$, otherwise there is nothing to show.

As noted, by \eqref{eq:from-persson-with-love}, we know that
\[
    \Sigma(\Omega) = \lim_{r\to \infty} \lambda(\Omega\setminus \overline{B_r}).
\]
Now fix $r>0$. Applying Lemma~\ref{lem:annulus-limit} to $\omega = \Omega \setminus \overline{B_r}$, we have $\lambda(\Omega\setminus \overline{B_r}) = \lim_{R\to \infty} \lambda(\Omega \cap A_{r,R})$, and we can find $R>0$ such that
\begin{equation}
|\lambda(\Omega\setminus \overline{B_r}) - \lambda(A_{r,R})| \le \frac{1}{r} .
\end{equation}
It follows directly that we can find a sequence of annular regions $\Omega_n := \Omega \cap A_{r_n, R_n}$ with $0 < r_n < R_n < r_{n+1}$ for all $n \in \N$, such that
\begin{equation}
    \Sigma(\Omega) = \lim_{n\to \infty} \lambda(\Omega_n).
\end{equation}

We now define
\begin{equation}
\label{eq:olympic-rings}
    \omega_i := \Omega_i \cup \Omega_{k+i} \cup \Omega_{2k+i} \cup \ldots = \bigcup_{j=0}^\infty \Omega_{jk+i}
\end{equation}
to be an infinite union of disjoint concentric ``rings'' in $\Omega$, for each $i=1,\ldots,k$. Then the $\omega_i$ are pairwise disjoint open sets, and thus form an admissible $k$-partition of $\Omega$. For each $i$, there are two possibilities: either $\infspec (\omega_i)$ is attained on some ring $\Omega_n$, or else
\begin{displaymath}
    \infspec (\omega_i) = \lim_{j \to \infty} \infspec(\Omega_{jk+i}) = \infess (\Omega).
\end{displaymath}
In the former case, the fact that $\infspec (\Omega_{jk+i})$ must be minimal at $n$ means that in this case
\begin{equation}
\label{eq:lordoftherings}
    \infspec (\omega_i) \le \lim_{j \to \infty} \infspec(\Omega_{jk+i}) = \infess (\Omega).
\end{equation}
Hence, either way, we have $\infspec (\omega_i) \le \infess (\Omega)$ for all $i$. In particular, this proves \eqref{eq:optimal-inequality-infty} (i.e. the case $p=\infty$).

We now consider the case $1 \le p < \infty$. Let $\partition := (\tilde\omega_1, \ldots, \tilde\omega_{k-1}) \in \partitionset{k-1}$ be an arbitrary ($k-1$)-partition. By \eqref{eq:to-infinity-and-beyond},
\begin{equation}
    \Lambda_{k-1,p}(\partition) = \lim_{r\to \infty} \Lambda_{k-1,p}(\tilde\omega_1 \cap B_r, \ldots, \tilde\omega_{k-1} \cap B_r).
\end{equation}
We fix $\varepsilon>0$ arbitrary and choose $r_0>0$ such that
\begin{displaymath}
    \Lambda_{k-1,p}(\tilde\omega_1 \cap B_{r_0}, \ldots, \tilde\omega_{k-1} \cap B_{r_0})^p <\Lambda_{k-1,p}(\partition)^p + \varepsilon.
\end{displaymath}
We then define $\omega_k := \Omega \setminus \overline{B_{r_0}}$; by Lemma~\ref{lem:domain-monotonicity}, we have that $\infspec (\omega_k) \le \infess (\Omega)$. It follows that 
\begin{displaymath}
\tilde{\mathcal{P}}_{r_0}:=(\tilde\omega_1 \cap B_{r_0}, \ldots, \tilde\omega_{k-1} \cap B_{r_0}, \omega_k)
\end{displaymath}
is a $k$-partition and satisfies
\begin{displaymath}
    \optenergy{k}{p}(\Omega)^p \leq \Lambda_{k,p}(\tilde{\mathcal{P}}_{r_0})= \Lambda_{k-1,p}(\tilde\omega_1 \cap B_{r_0}, \ldots, \tilde\omega_{k-1} \cap B_{r_0})^p +\lambda(\omega_k)<\Lambda_{k-1,p}(\partition)^p + \varepsilon + \infess (\Omega)^p.
\end{displaymath}
Since $\varepsilon>0$ and $(\omega_1,\ldots, \omega_{k-1})\in \partitionset{k-1}$ were arbitrary we infer that
\begin{equation}
    \optenergy{k}{p}(\Omega)^p \le 
    \optenergy{k-1}{p}(\Omega)^p +  \infess(\Omega)^p.
\end{equation}

Recalling that, by convention,  $\optenergy{0}{p}(\Omega)=0$, this also implies that
\begin{equation}
\label{eq:threshold-best-1}
    \optenergy{k-1}{p}(\Omega)^p +  \infess(\Omega)^p=\min_{\ell=0,\ldots, k-1} \{\optenergy{\ell}{p}(\Omega)^p + (k-\ell) \infess(\Omega)^p\}\leq k\infess(\Omega)^p. \qedhere
\end{equation}
\end{proof}

\section{Proof of Theorem~\ref{thm:secondmain}: Existence}
\label{sec:existence}

We will prove the existence result, Theorem~\ref{thm:secondmain}, in the following form:

\begin{theorem}
\label{thm:existence}
Given $1 \le p \le \infty$, suppose we have
\begin{equation}
\label{eq:compactness_condition_p}
    [\optenergyrel{k}{p} (\Omega)]^p < \min_{\ell=0,\ldots, k-1} \{  \optenergyrel{\ell}{p}(\Omega)^p +(k-\ell)\infess(\Omega)^p \}=\optenergyrel{k-1}{p}(\Omega)^p +\infess(\Omega)^p ,
\end{equation}
(if $1 \le p < \infty$), or
\begin{equation}
\label{eq:compactness_condition_infty}
    \optenergyrel{k}{\infty} (\Omega) < \infess (\Omega)
\end{equation}
(if $p = \infty$). Let $U_n = (u_{1,n},\ldots,u_{k,n}) \in H^1_{0,V} (\Omega)^k$ form a minimizing sequence for $\optenergyrel{k}{p} (\Omega)$, where each $u_{i,n}$ is $L^2$-normalized. Then there exists $U = (u_1,\ldots,u_k) \in H^1_{0,V}(\Omega{{\setminus \{0\}}})^k$ such that, up to a subsequence, $U_n {{\rightharpoonup}} U$ in $H^1_{0,V}(\Omega)$, $u_i\cdot u_j\equiv 0$ for all $i\neq j$, and
\begin{displaymath}
    \optenergyrel{k}{p} (\Omega) = \energyrel{k}{p} (U)>0.
\end{displaymath}
\txtr{
Moreover:
\begin{enumerate}
    \item[(a)] if $p=\infty$ and $\mathcal R_V(u_{i,n}) \to \optenergyrel{k}{\infty} (\Omega)$ for some $i=1,\ldots, k$, then $u_{i,n} \to u_i$ in $H^1_{0,V}(\Omega)$, and $\|u_i\|_2 = 1$. 
    \item[(b)] if $1 \le p < \infty$, then  $u_{i,n} \to u_i$ in $H^1_{0,V}(\Omega)$ and $\|u_i\|_2 = 1$ for every $i=1,\ldots, k$.
\end{enumerate}
}
\end{theorem}

The existence of such a minimizing sequence follows since $0 \le \optenergyrel{k}{p} (\Omega) < \infty$. In particular, the proposition does in fact yield the existence of a minimizer $U$ for $\optenergyrel{k}{p}(\Omega)$ (the extra statements beyond existence will be needed in Sections~\ref{sec:continuity} \txtr{ and~\ref{sec:regularity}; we note that the case $p=\infty$ is far more delicate}). The proof uses a cut-off argument based on functions with the following properties.

\begin{lemma}\label{lemma:existence_of_cutoff_functions}
Given $n\in \N$, there exist $\varphi_n,\psi_n\in C^\infty(\R^d)$ such that
\[
\varphi_n^2+\psi_n^2\equiv 1 \text{ a.e. in }\Omega,
\]
\[
\begin{cases}
\varphi_n=1,\ \psi_n=0 \text{ in } B_{n}\\
\varphi_n=0,\ \psi_n=1 \text{ in } \R^d\setminus B_{2n}\\
|\nabla \varphi_n|,\ |\nabla \psi_n| \leq C/n 
\end{cases} \qquad 
\]
\end{lemma}
\begin{proof}
Let $\varphi:\R^+\to \R$ be a $C^\infty$ function such that
\[
0\leq \varphi \leq 1,\ \varphi(x)=1 \text{ for $x\leq 1$}, \ \varphi(x)=0 \text{ for $x\geq 2$},
\]
and take $\psi(x):=\sqrt{1-\varphi^2(x)}$. Since $\varphi^2(x)=1+o(|x-1|^n)$ as $x\to 1$ for every $n\in \N$, then $\psi\in C^\infty(\R)$. 

It is now enough to take $\varphi_n(x):=\varphi(|x|/n)$, $\psi_n(x):=\psi(|x|/n)$.
\end{proof}

\begin{proof}[Proof of Theorem~\ref{thm:existence}]
 Fix $k \ge 1$ and $1 \le p \le \infty$ arbitrary, and let $U_n=(u_{1,n},\ldots, u_{k,n})\subset (H^1_{0,V}(\Omega))^k$ be a minimizing sequence for $\optenergyrel{k}{p} (\Omega)$, normalized in $L^2(\Omega)$, that is,
\[
u_{i,n}\cdot u_{j,n}\equiv 0\quad  \forall i\neq j,\ n\in \N,
\]
and
\begin{equation}\label{eq:minimizingsequence_p}
 \|u_{i,n}\|_{2}=1 \ \forall n\in \N,\ i=1,\ldots, k, \quad \text{with} \quad \energyrel{k}{p}(u_{n}) \to \optenergyrel{k}{p} (\Omega).
\end{equation}
By definition of $\optenergyrel{k}{p} (\Omega)$, the sequence $(U_n)$ is bounded in $(H^1_{0,V}(\Omega))^k$.

By the local compactness of the embedding $(H^1_{0,V}(\Omega))^k \hookrightarrow (L^2_{{\rm loc}}(\Omega))^k$, there exists some $U=(u_1,\ldots, u_k)\in (H^1_{0,V}(\Omega))^k$ such that, up to a subsequence, for each $i$,
\begin{equation}
\label{eq:convergence_to_u_p}
u_{i,n} \to u_i \quad \text{ weakly in } H^1_{0,V}(\Omega),\text{ strongly in }L^2_{{\rm loc}}(\Omega) \text{ and a.e. in } \Omega.
\end{equation}
We will show that $u$ minimizes $\optenergyrel{k}{p} (\Omega)$. The proof is divided into a number of steps, in some of which it is necessary to divide into the two cases $p=\infty$ and $1 \le p < \infty$, where the details differ slightly. Before starting, we recall the notation $\mathcal{R}_V (u)$ for the Rayleigh quotient of the function $u$, and $a_V(u)$ for the numerator, see \eqref{eq:rayleigh-quotient} and \eqref{eq:form}, respectively.

\emph{Step 1: Decomposition/cut-off of the $u_{i,n}$.} Given $n\in \N$, we take cut-off functions $\varphi_n,\psi_n\in C^\infty(\R^d)$ as in Lemma~\ref{lemma:existence_of_cutoff_functions}, that is, such that 
\[
\varphi_n^2+\psi_n^2\equiv 1 \text{ in }\R^d,
\]
while
\[
\begin{cases}
\varphi_n=1,\ \psi_n=0 \text{ in } B_{n}(0)=:B_n,\\
\varphi_n=0,\ \psi_n=1 \text{ in } \R^d\setminus B_{2n}(0)=: \R^d\setminus B_{2n}\\
|\nabla \varphi_n|,\ |\nabla \psi_n| \leq C/n.
\end{cases}
\]
Observe that, for all $m\in\N$, since $u_{i,m} \psi_n \in H^1_{0,V} (\Omega \setminus \overline{B_n})$,
\begin{equation*}
\liminf_{n\to \infty} \mathcal{R}_V (u_{i,m} \psi_n) \geq \Sigma(\Omega), 
\end{equation*}
as follows from Persson's theorem in the form of \eqref{eq:from-persson-with-love}. By a standard diagonal argument, there exists a subsequence $m_n$ such that, for every $i$,
\begin{equation}
\label{eq:part_escaping_p}
\liminf_{n\to \infty}\mathcal{R}_V (u_{i,m_n} \psi_n) \geq \Sigma(\Omega). 
\end{equation}

\emph{Step 2: Convergence of the functions $\varphi_n u_{i,n}$.} More precisely, we will show that, given $i$, up to a subsequence,
\begin{equation}
\label{eq:B_npart_p}
    \int_\Omega |\varphi_n u_{i,n}|^2 \to \int_\Omega u_i^2,
\end{equation}
and $\varphi_n u_{i,n}\rightharpoonup u_i$ weakly in $H^1_{0,V}(\Omega)$. Note that \eqref{eq:B_npart_p} also implies convergence a.e. in $\Omega$.

We start by checking the first statement. Indeed, given $m$, since the support of $\varphi_m$ is contained in $\overline{B_{2m}}$ and $u_{i,n}\to u_i$ in $L^2_{{\rm loc}} (\Omega)$, we have
\[
\int_\Omega |\varphi_m u_{i,n}|^2=\int_{B_{2m}} |\varphi_m u_{i,n}|^2 \to \int_{B_{2m}} |\varphi_m u_i|^2=\int_\Omega |\varphi_m u_i|^2
\]
as $n\to \infty$. On the other hand, by dominated convergence,
\[
\int_\Omega |\varphi_m u_i|^2 \to \int_\Omega |u_i|^2 \text{ as } m\to \infty.
\]
By a diagonal argument, we can choose $m\mapsto n_m$ increasing so that
\[
\left|\int_\Omega |\varphi_m u_{i,n_m}|^2-\int_\Omega |\varphi_m u_i|^2\right|\leq \frac{1}{m},
\]
which proves \eqref{eq:B_npart_p}. As for the second statement, $(\varphi_n u_{i,n})$ is a bounded sequence in $H^1_{0,V}(\Omega)$, hence there exists a weak limit $v_i\in H^1_{0,V}(\Omega)$ of the $\varphi_n u_{i,n}$. But the strong convergence $\varphi_n u_{i,n} \to u_i$ in $L^2(\Omega)$ and continuity of the embedding $H^1_{0,V} (\Omega) \hookrightarrow L^2(\Omega)$ imply that $v_i = u_i$ a.e. in $\Omega$.

\emph{Step 3: Decomposition of the form}. We have the following IMS localization formula (see also  \cite[Section~2]{Si82}):
\begin{equation}
\label{eq:exp_bilinearform_p}
    a_V(u_{i,n})=a_V(u_{i,n} \varphi_n)+a_V(u_{i,n} \psi_n)+O\left(\frac{1}{n^2}\right) \quad \text{ as } n\to \infty.
\end{equation}
Indeed,
\begin{align*}
    a_V(u_{i,n}\varphi_n)&=\int_\Omega (|\nabla (\varphi_n u_{i,n})|^2+V(x)|\varphi_n u_{i,n}|^2)\\
        &=\int_\Omega (|\nabla u_{i,n}|^2|\varphi_n|^2+ V(x) |\varphi_n|^2 |u_{i,n}|^2)+ \int_\Omega (|\nabla \varphi_n|^2 |u_{i,n}|^2+2u_{i,n} \varphi_n \nabla \varphi_n \cdot \nabla u_{i,n} ).
\end{align*}
Analogously,
\[
a_V(u_{i,n}\psi_n)=\int_\Omega (|\nabla u_{i,n}|^2|\psi_n|^2+ V(x) |\psi_n|^2 |u_{i,n}|^2)+ \int_\Omega (|\nabla \psi_n|^2 |u_{i,n}|^2+2u_{i,n} \psi_n \nabla \psi_n \cdot \nabla u_{i,n} ).
\]
The claim of this step now follows by adding $a_V(u_{i,n}\varphi)$ and $a_V(u_{i,n}\psi_n)$, together with the fact that $\varphi_n^2+\psi_n^2=1$, $\varphi_n\nabla \varphi_n+ \psi_n \nabla \psi_n=0$ and
\begin{equation}
\int_{\Omega} (|\nabla \varphi_n|^2 |u_{i,n}|^2 + |\nabla \psi_n|^2 |u_{i,n}|^2 ) \le \frac{C\max \{ \|\nabla \psi\|_\infty^2, \|\nabla \varphi\|_\infty^2\}}{n^2},
\end{equation}
since $(U_n)$ is a bounded sequence in $(H^1_{0,V}(\Omega))^k$.

\emph{Step 4: Nontriviality of the limit function $U$.} We will show that, for each $i$, $u_i \not\equiv 0$ in $\Omega$. Suppose by way of contradiction that, say, $u_1 \equiv 0$. By \eqref{eq:B_npart_p}, $u_{1,n} \varphi_n \to 0$ in $L^2(\Omega)$ and so, since
\[
    1=\|u_{1,n}\|_2^2=\|u_{1,n} \varphi_n\|_2^2+\|u_{1,n} \psi_n\|_2^2
\]
it follows that
\begin{equation}
\label{eq:step4-badlimit}
    \|u_{{1},n}\psi_n\|_2^2\to 1,
\end{equation}
whence
\begin{displaymath}
\begin{aligned}
    \liminf_{n \to \infty} \mathcal{R}_V (u_{1,n})
    &= \liminf_{n \to \infty} \frac{a_V(u_{1,n}\varphi_n)+a_V(u_{1,n}\psi_n)+O(1/n^2)}{\|u_{1,n}\varphi_n\|_2^2+\|u_{1,n} \psi_n\|_2^2}\\
    &\ge \liminf_{n \to \infty} \frac{a_V(u_{1,n} \psi_n)}{\|u_{1,n} \psi_n\|_2^2} = \liminf_{n \to \infty} \mathcal{R}_V(u_{1,n} \psi_n) \ge \infess (\Omega),
\end{aligned}
\end{displaymath}
where for the first equality we have \eqref{eq:exp_bilinearform_p}, and the inequality follows from the positivity of the form, $a_V(u_{1,n}\varphi_n) \ge 0$, together with \eqref{eq:step4-badlimit}.

We will show that this is a contradiction to \eqref{eq:compactness_condition_infty} (case $p=\infty$) and \eqref{eq:compactness_condition_p} (case $1\le p < \infty$), respectively. Indeed, if $p=\infty$, then by \eqref{eq:compactness_condition_infty},
\begin{equation}
\label{eq:condition_contradiction_infty}
    \infess(\Omega) > \optenergyrel{k}{\infty}(\Omega) = \lim_{n \to \infty} \max_{i} \mathcal{R}_V (u_{i,n})
    \ge \liminf_{n\to\infty} \mathcal{R}_V(u_{1,n}) \ge \infess (\Omega),
\end{equation}
a contradiction.

For the case $1 \le p < \infty$, note that, for any $n \in \N$, $(u_{2,n},\ldots,u_{k,n})$ is a valid $(k-1)$-tuple for $\optenergyrel{k-1}{p} (\Omega)$, so that
\begin{displaymath}
    \sum_{i=2}^k \mathcal{R}_V(u_{i,n})^p \ge \optenergyrel{k-1}{p}(\Omega)^p.
\end{displaymath}
Hence, by \eqref{eq:compactness_condition_p},
\begin{displaymath}
\begin{aligned}
    \optenergyrel{k-1}{p}(\Omega)^p + \infess(\Omega)^p
    > \optenergyrel{k}{p}(\Omega)^p &= \lim_{n\to\infty}\sum_{i=1}^k \mathcal{R}_V (u_{i,n})^p\\
    &\ge \liminf_{n\to\infty}\mathcal{R}_V(u_{1,n}) + \optenergyrel{k-1}{p}(\Omega)^p \ge \infess(\Omega) + \optenergyrel{k-1}{p}(\Omega)^p,
\end{aligned}
\end{displaymath}
a contradiction.

\emph{Step 5: Convergence of the smaller Rayleigh quotients from the decomposition to the infimal value.} We will prove that
\begin{equation}
\label{eq:conv_of_min}
    \lim_{n\to\infty}\ | (\min \left\{ \mathcal{R}_V (u_{i,n} \varphi_n),\mathcal{R}_V (u_{i,n} \psi_n) \right\})_i |_p = \optenergyrel{k}{p} (\Omega),
\end{equation}
where the norm is understood as the $\ell^p$-norm of the vector 
\begin{multline*}
(\min \left\{ \mathcal{R}_V (u_{i,n} \varphi_n),\mathcal{R}_V (u_{i,n} \psi_n) \right\})_i \\= (\min \left\{ \mathcal{R}_V (u_{1,n} \varphi_n),\mathcal{R}_V (u_{1,n} \psi_n) \right\},\ldots, \min \left\{ \mathcal{R}_V (u_{k,n} \varphi_n),\mathcal{R}_V (u_{k,n} \psi_n) \right\}) \in \R^k.
\end{multline*}

Observe that $u_{i,n}\varphi_n \cdot u_{j,n}\varphi_n\equiv 0$, $u_{i,n}\psi_n \cdot u_{j,n}\varphi_n\equiv 0$ and $u_{i,n}\psi_n \cdot u_{j,n}\psi_n\equiv 0$ for every $i\neq j$. Therefore, by definition of $\optenergyrel{k}{p} (\Omega)$, using either $u_{i,n}\varphi_n$ or $u_{i,n}\psi_n$ as the test function in the $i$-th position of the test $k$-tuple, depending on which has the lower Rayleigh quotient, we have
\[
\optenergyrel{k}{p} (\Omega) \leq | (\min\{\mathcal{R}_V (u_{i,n} \varphi_n), \mathcal{R}_V (u_{i,n} \psi_n)\})_i|_p,
\]
and so
\[
\optenergyrel{k}{p} (\Omega)\leq \liminf_{n\to\infty}\ |(\min \left\{ \mathcal{R}_V (u_n \varphi_n),\mathcal{R}_V (u_n \psi_n) \right\})_i|_p.
\]
On the other hand, using the elementary inequality
\[
\min\left\{\frac{a}{c},\frac{b}{d}\right\}\leq \frac{a+b}{c+d} \text{ for every } a,b,c,d\geq 0, c,d\neq 0,
\]
and using also \eqref{eq:exp_bilinearform_p},
\begin{align*}
    \min \left\{ \mathcal{R}_V (u_{i,n} \varphi_n),\mathcal{R}_V (u_{i,n} \psi_n) \right\}
    &=\min \left\{ \frac{a_V(u_{i,n} \varphi_n)}{\|u_{i,n} \varphi_n\|_2^{2}},\frac{a_V(u_{i,n} \psi_n)}{\|u_{i,n} \psi_n\|_2^{2}} \right\}\\
    &\le \frac{a_V(u_{i,n} \varphi_n)+a_V(u_{i,n} \psi_n)}{\|u_{i,n} \varphi_n\|_2^{2}+\|u_{i,n} \psi_n\|_2^{2}}\\
    &=\frac{a_V(u_{i,n})+O(1/n^2)}{\|u_{i,n}\|_2^2}=\mathcal{R}_V(u_{i,n}) + O(1/n^2).
\end{align*}
Taking the $p$-norm over $i=1,\ldots,k$ and passing to the limit yields
\begin{align*}
    \limsup_{n\to\infty}\ |(\min \left\{ \mathcal{R}_V (u_n \varphi_n),\mathcal{R}_V (u_n \psi_n) \right\})_i|_p
    &\le \limsup_{n\to\infty}\ |(\mathcal{R}_V(u_{i,n}) + O(1/n^2))_i|_p \\
    &=\limsup_{n\to\infty}\ |(\mathcal{R}_V(u_{i,n}))_i|_p = \optenergyrel{k}{p} (\Omega),
\end{align*}
and \eqref{eq:conv_of_min} is proved.

\emph{Step 6: The function $u_{i,n}\varphi_n$ supported in the interior has the lower Rayleigh quotient in the limit.} We will prove that, for every $i=1,\ldots, k$ and large $n$,
\begin{equation}
\label{eq:finallyn}
    \min \left\{ \mathcal{R}_V (u_{i,n} \varphi_n),\mathcal{R}_V (u_{i,n} \psi_n) \right\}=\mathcal{R}_V (u_{i,n} \varphi_n).
\end{equation}
Assume that this is not the case for some $i$, without loss of generality $i=1$, that is, that $\min \left\{ \mathcal{R}_V (u_{1,n} \varphi_n),\mathcal{R}_V (u_{1,n} \psi_n) \right\}=\mathcal{R}_V (u_{1,n} \psi_n)$ for large $n$.

Observe that 
\[
    \liminf_{n\to\infty} \mathcal{R}_V(u_{1,n}\psi_n) \geq \Sigma(\Omega)
\]
by \eqref{eq:part_escaping_p}, which already yields a contradiction to \eqref{eq:compactness_condition_infty} in the case $p=\infty$, since then
\[
    \infess(\Omega) > \optenergyrel{k}{\infty}(\Omega) = \lim_{n\to \infty}| (\min \left\{ \mathcal{R}_V (u_{i,n} \varphi_n),\mathcal{R}_V (u_{i,n} \psi_n) \right\})_i |_\infty \ge \liminf_{n\to\infty} \mathcal{R}_V(u_{1,n}\psi_n) \geq \Sigma(\Omega).
\]
If $1\le p < \infty$, then by definition,
\[
\sum_{i=2}^k \left(\min \left\{ \mathcal{R}_V (u_{i,n} \varphi_n),\mathcal{R}_V (u_{i,n} \psi_n) \right\}\right)^p \geq \optenergyrel{k-1}{p} (\Omega)^p,
\]
leading to
\begin{equation*}
\optenergyrel{k}{p} (\Omega)^p=\lim_n \sum_{i=1}^k \left(\min \left\{ \mathcal{R}_V (u_{i,n} \varphi_n),\mathcal{R}_V (u_{i,n} \psi_n) \right\}\right)^p\geq \Sigma(\Omega)^p +\optenergyrel{k-1}{p} (\Omega)^p,
\end{equation*}
a contradiction to \eqref{eq:compactness_condition_p}.

\emph{Step 7: The limit function $U$ is a minimizer.} We can now show that $U = (u_1,\ldots,u_k)$ is indeed a minimizer: we have
\begin{displaymath}
\begin{aligned}
    \optenergyrel{k}{p} (\Omega) &\le |(\mathcal{R}_V(u_i))_i|_p = \left| \left(\frac{a_V(u_i)}{\|u_i\|_2^2}\right)_i \right|_p\\
    &\le \liminf_{n\to\infty}\ \left|\left( \frac{a_V(u_{i,n}\varphi_n)}{\|u_{i,n}\varphi_n\|_2^2} \right)_i \right|_p = \liminf_{n\to\infty}\ |(\mathcal{R}_V(u_{i,n}\varphi_n))_i|_p\\
    &=\liminf_{n\to\infty}\ |(\min\{\mathcal{R}_V(u_{i,n}\varphi_n),\mathcal{R}_V(u_{i,n}\psi_n)\})_i|_p =\optenergyrel{k}{p} (\Omega),
\end{aligned}
\end{displaymath}
where we have used, respectively, the definition of $\optenergyrel{k}{p}(\Omega)$ in the first step, the fact that $u_i\not\equiv 0$ (Step 4) in the second, Step 2 together with Fatou's lemma in the third, the definition of the Rayleigh quotient in the fourth, \eqref{eq:finallyn} (Step 6) for the penultimate equality, and \eqref{eq:conv_of_min} (Step 5) for the last equality.

\emph{Step 8: The case $1 \leq p < \infty$: Convergence of the sequence in $H^1_{0,V}$.} Now suppose that $1 \leq p < \infty$. We already saw, at the beginning of the proof, that $U_n \rightharpoonup U$ weakly in $H^1_{0,V} (\Omega)$, and hence also weakly in $L^2(\Omega)$ (for any $1 \leq p < \infty)$; we also know (Steps 5 and 6) that, for any $i=1,\ldots,k$,
\begin{equation}
\label{eq:rq-conv}
    \mathcal{R}_V (u_{i,n}) \to \mathcal{R}_V (u_i).
\end{equation}
Furthermore, by Step 7, we have that for a subsequence $\varphi_n u_{i,n} \to u_i$ in $L^2$. This, together with the convergence of the respective Rayleigh quotients as a whole (Steps 5 and 6) implies that
\begin{displaymath}
a_V(\varphi_n u_{i,n}) \to a_V(u_i).
\end{displaymath}
 We want to show that $\|u_i\|_2=1$ for each $i$, this plus \eqref{eq:rq-conv} will then imply that $a_V(u_{i,n}) \to a_V(u_i)$. Assume for a contradiction that $0< \| u_1\|_2 <1$, we will show that this is a contradiction to \eqref{eq:compactness_condition_p} (case $1\le p < \infty$), respectively. Indeed, 
\begin{displaymath}
    \mathcal{R}_V(u_i)^p + \optenergyrel{k-1}{p}(\Omega)^p \le \mathcal{R}_V(u_i)^p + \sum_{\substack{j=1\\ j\neq i}} \mathcal{R}_V(u_j)^p = \optenergyrel{k}{p}(\Omega)^p < \Sigma(\Omega)^p + \optenergyrel{k-1}{p}(\Omega)^p 
\end{displaymath}
and thus we have
\begin{equation}
\label{eq:below!}
\mathcal R_V(u_i) < \infess(\Omega)
\end{equation}
for all $i=1, \ldots, k$. We conclude
\begin{displaymath}
\begin{aligned}
\optenergyrel{k}{p}(\Omega)&= \lim_{n\to \infty} |(\mathcal{R}_V(u_{i,n}))_i|_p \\
&= \lim_{n\to \infty} |(a_V(\varphi_n u_{i,n}) + a_V(\psi_n u_{i,n}))_i|_p \\
&\ge |(\mathcal{R}_V(u_i) \|u_i\|^2_2 + \infess(\Omega) (1- \|u_{i}\|_{2}^2)|_p\\
&> |(\mathcal{R}_V(u_i))_i|_p = \optenergyrel{k}{p}(\Omega).
\end{aligned}
\end{displaymath}
We have shown that $U_n \rightharpoonup U$ weakly in $H^1_{0,V}(\Omega)$ and, thanks to the fact that $\|u_{i,n}\|_2 \to \|u_i\|_2$ and $a_V(u_{i,n}) \to a_V(u_i)$ for each $i$, that $\|U_n\|_{H^1_{0,V}} \to \|U\|_{H^1_{0,V}}$. This implies that $U_n \to U$ strongly in $H^1_{0,V}(\Omega)^k$.

\txtr{
\emph{Step 9: The case $p=\infty$: Convergence of the Rayleigh quotients of the maximizing $u_i$.} Now suppose that $p=\infty$, and choose any $i=1,\ldots,k$ for which $R_V(u_{i,n})\to \optenergyrel{k}{\infty}(\Omega)$. Since $R_V(u_i)\leq \liminf R_V(u_{i,n})$, then  $\mathcal{R}_V(u_i) = \max_i \{\mathcal{R}_V(u_i)\} = \optenergyrel{k}{\infty}(\Omega)$.  We wish to show that $\|u_i\|_2 = 1$. If this is not the case, that is, if $0 < \|u_i\|_2 < 1$, then by \eqref{eq:compactness_condition_infty}, we have 
\begin{displaymath}
\begin{aligned}
\optenergyrel{k}{\infty}(\Omega)&= \max_j \{\mathcal{R}_V(u_j)\} = \lim_{n\to \infty} \max_j \{\mathcal{R}_V(u_{j,n})\} \\
&= \lim_{n\to \infty} \max_j \left\{a_V(\varphi_n u_{j,n}) + a_V(\psi_n u_{j,n}) \right\}\\
&\ge \mathcal{R}_V(u_i) \|u_i\|^2_2 + \infess(\Omega) (1- \|u_{i}\|_{2}^2)\\
&> \optenergyrel{k}{\infty}(\Omega),
\end{aligned}
\end{displaymath}
where, for the third equality, we have used Step 3. As in Step 8, this also allows us to conclude that $u_{i,n} \to u_i$ in $H^1_{0,V} (\Omega)$ for this $i$, since then trivially $\|u_{i,n}\|_2 \to \|u_i\|_2$, and $a_V(u_{i,n}) \to a_V(u_i)$.
}
\end{proof}

\section{Continuity and monotonicity in $p$}
\label{sec:continuity}

\txtr{In this section we prove the monotonicity and continuity statements for the weak versions of the energies in Proposition~\ref{prop:p}, that is, parts (1) and (2), for $\optenergyrel{k}{p}(\Omega)$ and $\thresholdrel{k}{p}(\Omega)$, since these statements will be needed for our study of the regularity of the minimizers. The equalities $\optenergyrel{k}{p}(\Omega) = \optenergy{k}{p}(\Omega)$ and $\thresholdrel{k}{p}(\Omega) = \threshold{k}{p}(\Omega)$, as well as part (3), will be treated afterwards.

\begin{proposition}\label{prop:p-continuity}
Let $k\ge 1$. Then:
\begin{itemize}
    \item[(1)] The functions
    \begin{equation*}
    p \mapsto \optenergyrel{k}{p} (\Omega) \quad \text{ and }\quad p \mapsto  \thresholdrel{k}{p}(\Omega)
\end{equation*}
are continuous and non-increasing in $p \in [1,\infty]$.
\end{itemize}
Now let $p \in [1,\infty]$ and suppose the threshold condition \eqref{eq:strict-p} is satisfied for this $p$ (in particular, by Theorem \ref{thm:secondmain}, a minimizer exists for $\optenergyrel{k}{p}(\Omega)$). Then:
\begin{itemize}
\item[(2)] There exists a neighborhood of $p \in [1,\infty]$ in which \eqref{eq:strict-p} holds. 
\end{itemize}
\end{proposition}
}
\begin{proof}
(1) We will show that $p\mapsto \optenergyrel{k}{p}(\Omega)$ is non-increasing and continuous.     Suppose $1 \le p_1 \le p_2 < \infty$. A consequence of H\"older's  inequality in $\ell^p$-spaces is
    \begin{displaymath}
        |\cdot|_\infty\le |\cdot|_{p_2} \le |\cdot|_{p_1} \le k^{\tfrac{1}{p_1} - \tfrac{1}{p_2}} |\cdot|_{p_2} \le k^{\tfrac{1}{p_1}} |\cdot|_\infty .
    \end{displaymath}
    This implies in particular that, for any given $U=(u_1,\ldots,u_k) \in (H^1_{0,V}(\Omega) \setminus \{0\})^k$, the function
    \begin{displaymath}
        p \mapsto \energyrel{k}{p} (u_1,\ldots,u_k) = \left| (\mathcal{R}_V(u_1), \ldots, \mathcal{R}_V(u_k))\right|_p
    \end{displaymath}
    is non-increasing in $p \in [1,\infty]$, being also continuous. The claim now follows from the characterization of $\optenergyrel{k}{p}(\Omega)$ as the infimum over all such $k$-tuples in $(H^1_{0,V}(\Omega) \setminus \{0\})^k$. More precisely, suppose $p\to p_0$ with $p_0 \in [1,\infty)$, then for $U=(u_1, \ldots, u_k) \in (H^1_{0,V}(\Omega))^k$ we have
\begin{displaymath}
     \min\{k^{\frac{1}{p_0}- \frac{1}{p}}, 1\} |(\mathcal R_V(u_i))|_{p_0} \le |(\mathcal R_V(u_i))|_{p} \le \max\{k^{\frac{1}{p} - \frac{1}{p_0}}, 1\} |(\mathcal R_V(u_i))|_{p_0}.  
\end{displaymath}
For $p_0 = \infty$, then we have for $p\in[1,\infty)$ respectively
\begin{displaymath}
|(\mathcal R_V(u_i))|_\infty\le |(\mathcal R_V(u_i))|_p \le k^{\frac{1}{p}} |(\mathcal R_V(u_i))|_\infty
\end{displaymath}

Then, since $\energyrel{k}{p}(U)=|(\mathcal R_V(u_i))|_p$, passing to the infimum we get, for $p_0 \in [1, \infty)$,
\begin{displaymath}
     \min\{k^{\frac{1}{p_0}- \frac{1}{p}}, 1\} \optenergyrel{k}{p_0}(\Omega) \le \optenergyrel{k}{p}(\Omega) \le \max\{k^{\frac{1}{p} - \frac{1}{p_0}}, 1\} \optenergyrel{k}{p_0}(\Omega).  
\end{displaymath}
and, for $p_0 = \infty$,
\begin{displaymath}
\optenergyrel{k}{\infty}(\Omega)\le \optenergyrel{k}{p}(\Omega) \le k^{\frac{1}{p}} \optenergyrel{k}{\infty}(\Omega).
\end{displaymath}
Then we have monotonicity in $p$ and, as $p\to p_0$, 
\begin{displaymath}
    \lim_{p \to p_0} \optenergyrel{k}{p}(\Omega) = \optenergyrel{k}{p_0} (\Omega).
\end{displaymath}

The monotonicity of $\thresholdrel{k}{p}(\Omega) = \left(\optenergyrel{k-1}{p}(\Omega)^p + \infess(\Omega)^p\right)^{1/p}$ for $1 \leq p < \infty$ is a consequence of the fact that $p \mapsto \optenergyrel{k-1}{p}(\Omega)$ is a positive, non-increasing function and that $p \mapsto |\cdot |_p$ is likewise non-increasing.

The continuity of $p\mapsto \thresholdrel{k}{p}(\Omega)$ for $1\le p < \infty$ is an immediate consequence of the continuity of $p \mapsto \optenergyrel{k}{p}(\Omega)$. 
For the case $p=\infty$ note that, for any $1 \leq q < \infty$,
\[
    \thresholdrel{k}{\infty}(\Omega)\le \thresholdrel{k}{q}(\Omega) \le k^{1/q}\infess (\Omega).
\]
Letting $q \to \infty$,  we conclude that $\thresholdrel{k}{q}(\Omega) \to \thresholdrel{k}{\infty}(\Omega)$, and we have continuity for all $ 1\le p \le \infty$. This also establishes monotonicity at $p=\infty$.

(2) Suppose that, for some $p_0 \in [1,\infty]$, the threshold condition \eqref{eq:strict-p} holds. Then by continuity of the quantity $\optenergyrel{k}{p}(\Omega)$ and the continuity of the threshold expression itself in $p$ (1), there exists a neighborhood $I \subset [1,\infty]$ of $p_0$ such that \eqref{eq:strict-p} holds for all $p \in I \setminus \{p_0\}$.

\end{proof}

\section{Proof of Theorem~\ref{thm:regularity}: Regularity}
\label{sec:regularity}

The goal of this section is  to prove Theorem~\ref{thm:regularity}, on the regularity of the minimizers of $\optenergyrel{k}{p}(\Omega)$.

We will need the following auxiliary result, which has been used for instance in  \cite{ASST,CoTeVe05,HHT09}. For a complete proof, see \cite[Lemma A.1]{ASST}.

\begin{lemma}\label{lemma:expansion_L^2} Let $u \in L^2(\Omega)$ with $u^+\not\equiv 0$. Then, for all $\varphi \in L^2(\Omega)$, 
\begin{equation}\label{inverse-norm-lemma}
\displaystyle \frac{1}{\|(u \pm t\varphi)^{+}\|_2^2 }= \dfrac{1}{\|u^+\|_2^2} \mp \frac{2t}{\|u^+\|_2^4} \int_{\Omega}u^+\varphi + O(t^2)\,  \|\varphi\|_2^2 \,\qquad \text{ as } t\to 0^+,\vspace{10pt}
\end{equation}
where the bound of $\frac{O(t^2)}{t^2}$ depends only on $\|u^+\|_ 2$, as $t\to 0^+$.
\end{lemma}

 \begin{proposition}
 \label{prop:diff-ineq-p}
   Take $1\le p<\infty$ and  assume there exists a minimizer $(u_1,\ldots, u_k)$ for $\optenergyrel{k}{p}(\Omega)$, normalized in $L^2(\Omega)$. For each $i=1,\ldots, k$, take
     \[
    v_i=a_{i} u_i,\quad \text{ for } \quad a_i^2=a_{i,p}^2:=\frac{\mathcal{R}_V(u_i)^{p-1}}{\optenergyrel{k}{\infty}(\Omega)^{p-1}}\in \mathbb{R}\setminus\{0\}. 
     \]
     The following differential inequalities are satisfied in the distributional sense:
     \begin{enumerate}
     \item $-\Delta v_i + V(x) v_i \leq \mathcal{R}_V(v_i) v_i$;
     \item $-\Delta\left(v_i-\sum_{j \neq i} v_j\right) + V(x) \left(v_i-\sum_{j \neq i} v_j\right) \geq \mathcal{R}_V(v_i) v_i-\sum_{j \neq i} \mathcal{R}_V(v_j) v_j$.
     \end{enumerate}
      \end{proposition}

      \begin{proof}
This proof is mostly taken from \cite{CoTeVe05, HHT09}.       We point out that in \cite{CoTeVe05} the domain $\Omega$ is bounded and $V\equiv 0$, which is not the case in our situation. However, all the arguments apply almost word-for-word; we sketch them here, following the approach of the proof of \cite[Lemma 3.11]{HHT09}. From now on, we let $\varphi \in C_c^{\infty}(\Omega)$ be a nonnegative function. Without loss of generality, we prove the inequalities for $i=1$.

      \smallbreak 
      
      \noindent Proof of (1).  Consider, for $t>0$ small, the perturbation
      \[
      u_{1,t}:=\left(u_1-t \varphi\right)^{+}.
      \]
Observe that $u_{1,t}\in H^1_{0,V}(\Omega)\setminus \{0\}$, and that $u_{1,t}\cdot u_j\equiv 0$ whenever $j\geq 2$. Then
\begin{equation}\label{eq:S_aux1}
\energyrel{k}{p}(u_1,\ldots, u_k)=\optenergyrel{k}{p}(\Omega)\leq \energyrel{k}{p}(u_{1,t},\ldots, u_k)=\left( \mathcal{R}_V(u_{1,t})^p + \sum_{j\geq 2} \mathcal{R}_V(u_j)^p \right)^{1/p}.
\end{equation}
Using Lemma \ref{lemma:expansion_L^2} and the fact that $\|u_1\|_2=1$,
\begin{align*}
\mathcal{R}_V(u_{1,t})&=\frac{\int_\Omega |\nabla \left(u_1-t \varphi\right)^{+}|^2 + V(x)|\left(u_1-t \varphi\right)^{+}|^2}{\|u_{1,t}\|_2^2}\\
				&\leq \left(\int_\Omega |\nabla \left(u_1-t \varphi\right)|^2+V(x)\left(u_1-t \varphi\right)^2\right)\left(1+2t\int_\Omega u_1\varphi + o(t)\right)\\
				&=\left(\int_\Omega (|\nabla u_1|^2 + V(x) u_1^2) - 2t \int_\Omega (\nabla u_1\cdot \nabla \varphi + V(x) u_1 \varphi ) + o(t)\right) \left(1+2t\int_\Omega u_1\varphi + o(t)\right)\\
				&=\mathcal{R}_V(u_1)-2t \int_\Omega (\nabla u_1\cdot \nabla \varphi + V(x) u_1\varphi) + 2t \mathcal{R}_V(u_1)\int_\Omega u_1 \varphi + o(t).
\end{align*}

\smallbreak

Let $f(t)=\left(t^p + \sum_{j\geq 2} \mathcal{R}_V(u_j)^p\right)^{1/p}$. For $a,b\in \R$ there exists $\xi$ between $a$ and $b$ such that
\begin{align*}
f(b)&=f(a)+f'(a)(b-a)+\frac{f''(\xi)}{2}(b-a)^2
\end{align*}
with 
\[
f'(t)=\frac{t^{p-1}}{\left(t^p + \sum_{j\geq 2}a_j^p \right)^\frac{p-1}{p}},\quad f''(t)=\frac{(p-1)t^{p-2}}{\left(t^p + \sum_{j\geq 2}a_j^p \right)^\frac{p-1}{p}}+\frac{(1-p)t^{2p-2}}{\left(t^p + \sum_{j\geq 2}a_j^p \right)^\frac{2p-1}{p}}.
\]
We apply this Taylor expansion with
\[
a=\mathcal{R}_V(u_1),\qquad b=\mathcal{R}_V(u_1)-2t \int_\Omega (\nabla u_1\cdot \nabla \varphi + V(x)u_1\varphi)+2t \mathcal{R}_V(u_1)\int_\Omega u_1 \varphi +o(t).
\]
In this case, the term $\frac{f''(\xi)}{2}(b-a^2)$ is an $o(t)$, as $t\to 0$. So, in conclusion,
\begin{align*}
&\energyrel{k}{p}(u_1,\ldots, u_k)\leq \energyrel{k}{p}(u_{1,t},u_2,\ldots, u_k)=f(\mathcal{R}_V(u_{1,t})) \\
&\leq f\left(\mathcal{R}_V(u_1)-2t \int_\Omega (\nabla u_1\cdot \nabla \varphi + V(x)u_1\varphi)+2t \mathcal{R}_V(u_1)\int_\Omega u_1 \varphi +o(t)\right)\\
	&= f(\mathcal{R}_V(u_1))+f'(\mathcal{R}_V(u_1))\left(-2t \int_\Omega (\nabla u_1\cdot \nabla \varphi + V(x)u_1\varphi)+2t \mathcal{R}_V(u_1)\int_\Omega u_1 \varphi +o(t)\right) + o(t)\\
	&=\energyrel{k}{p}(u_1,\ldots, u_k)+\\ &\qquad \frac{\mathcal{R}_V(u_1)^{p-1}}{\energyrel{k}{p}(u_1,\ldots, u_k)^{p-1}}\left(-2t \int_\Omega (\nabla u_1\cdot \nabla \varphi + V(x)u_1\varphi)+2t \mathcal{R}_V(u_1)\int_\Omega u_1 \varphi +o(t)\right) + o(t)
\end{align*}
From this, we deduce
\[
0\leq \frac{\mathcal{R}_V(u_1)^{p-1}}{\energyrel{k}{p}(u_1,\ldots, u_k)^{p-1}}\left(-2t \int_\Omega (\nabla u_1\cdot \nabla \varphi + V(x)u_1\varphi)+2t \mathcal{R}_V(u_1)\int_\Omega u_1 \varphi +o(t)\right) + o(t)
\]
Dividing the result by $t>0$, and letting $t\to 0$, implies $-\Delta u_1+V(x)u_1\leq \mathcal{R}_V(u_1) u_1$, which is equivalent to (1).

\medbreak

\noindent Proof of (2). For convenience of the reader, we present here a sketch of the proof, which follows the one of item (2) of \cite[Lemma 3.11]{HHT09}, taking therein the choice $\Lambda:=\optenergyrel{k}{\infty}^\frac{p-1}{p}$ and not considering the correction terms present there. Consider the deformation
\[
\left(v_{1, t}, \ldots, v_{k, t}\right):=\left(\left(\hat{v}_1+t \varphi\right)^{+},\left(\hat{v}_2-t \varphi\right)^{+}, \ldots,\left(\hat{v}_k-t \varphi\right)^{+}\right)
  \]
  where, for each $i$, $\hat v_i:=v_i-\sum_{j\neq i} v_j$.
  Observe that $v_{i,t}\in H^1_{0,V}(\Omega)\setminus \{0\}$, and that $v_{i,t}\cdot v_{j,t}\equiv 0$ whenever $j\neq i$. Moreover, $\sum_{j=1}^k v_{j,t}=\hat{v}_1+t\varphi$ a.e. in $\Omega$.   We have
  \begin{align}\label{eq:class_S_aux}
  0&\leq \energyrel{k}{p}\left(v_{1, t}, \ldots, v_{k, t}\right)-\energyrel{k}{p}(v_1,\ldots, v_k)\\
  &=\sum_{i=1}^k\left(\frac{\mathcal{R}_V(v_i)}{\optenergyrel{k}{p}(\Omega)}\right)^{p-1}\left(\mathcal{R}_V(v_{i,t})-\mathcal{R}_V(v_i)\right)+O(t^2).
  \end{align}
The main difficulty in this proof is that we can not perform a Taylor expansion of each term $\|\nabla v_{i,t}\|_2^2$ individually; however, we can expand its sum in $i$. Define
  \[
  \delta_{i,t}(\varphi):=\frac{1}{t}\int_\Omega \left (|\nabla v_{i,t}|^2-|\nabla v_i|^2+V(x)(v_{i,t}^2-v_i^2)\right),\qquad i=1,\ldots, k.
  \]
  We have
  \begin{align*}
  \mathcal{R}_V(v_{1,t})-\mathcal{R}_V(v_1)&=\int_\Omega (|\nabla v_{1,t}|^2+V(x)|v_{1,t}|^2)\left(\frac{1}{\|v_1\|_2^2}-\frac{2t}{\|v_1\|_2^4}\int_\Omega v_1\varphi+o(t)\right)-\mathcal{R}_V(v_1)\\
  &=\frac{t\delta_{1,t}(\varphi)}{a_1^2}-\frac{2t}{a_1^4}\int_\Omega (|\nabla v_{1,t}|^2+V(x)v_{1,t}^2)\int_\Omega v_1 \varphi+o(t)\\
  &=\frac{t\delta_{1,t}(\varphi)}{a_1^2}-\frac{2t}{a_1^2}R_V(v_1)\int_\Omega v_1\varphi-\frac{2t^2\delta_{1,t}(\varphi)}{a_1^4}\int_\Omega v_1\varphi + o(t);
  \end{align*}
  analogously, for $i\geq 2$,
  \[
  \mathcal{R}_V(v_{i,t})-\mathcal{R}_V(v_i) =\frac{t\delta_{i,t}(\varphi)}{a_i^2}+\frac{2t}{a_i^2}\mathcal{R}_V(v_i)\int_\Omega v_i\varphi+\frac{2t^2\delta_{i,t}(\varphi)}{a_i^4}\int_\Omega v_i\varphi + o(t).
  \]
  Going back to \eqref{eq:class_S_aux}, this yields
  \begin{align*}
  0\leq &t\optenergyrel{k}{\infty}(\Omega)^{p-1}\left(\sum_{i=1}^k \delta_{i,t}(\varphi)-2\mathcal{R}_V(u_1)\int_\Omega v_1\varphi + \sum_{i\geq 2} 2\mathcal{R}_V(u_i)\int_\Omega v_i \varphi \right)\\
  &+2t^2 \optenergyrel{k}{\infty}(\Omega)^{p-1}\left(\frac{\delta_{1,t}(\varphi)}{a_1^2}\int_\Omega u_1\varphi - \sum_{i\geq 2} \frac{\delta_{i,t}(\varphi)}{a_i^2}\int_\Omega u_i \varphi\right)+o(t)\\
  = &t\optenergyrel{k}{\infty}(\Omega)^{p-1}\left(\sum_{i=1}^k \delta_{i,t}(\varphi)-2\mathcal{R}_V(u_1)\int_\Omega v_1\varphi + \sum_{i\geq 2} 2\mathcal{R}_V(u_i)\int_\Omega v_i \varphi \right)+o(t).
  \end{align*}
  Combining this with
  \begin{align*}
  \sum_{i=1}^k \delta_{i,t}(\varphi)&=\frac{1}{t}\left(\int_\Omega (|\nabla (\hat v_1+t\varphi)|^2+V(x)(\hat v_1+t\varphi)^2)-\sum_{i=1}^k (|\nabla v_i|^2+V(x)v_i^2)\right)\\
  &=2\int_\Omega (\nabla \hat v_1\cdot\nabla \varphi + V(x)\hat v_1 \varphi)+o(t),
  \end{align*}
  we conclude that
  \[
  0\leq 2t \optenergyrel{k}{\infty}(\Omega)^{p-1} \left(\int_\Omega (\nabla \hat v_1\cdot\nabla \varphi + V(x)\hat v_1 \varphi)-\mathcal{R}_V(u_1)\int_\Omega v_1 \varphi+\sum_{i\geq 2} \mathcal{R}_V(u_i)\int_\Omega v_i \varphi \right)+o(t).
  \]
  By dividing by $t>0$ and letting $t\to 0$, we conclude the proof.
      \end{proof}

Given a bounded domain $A$, we consider the following set introduced in \cite{CTV_JFA03,CTV_Indiana05,CoTeVe05}:
\begin{multline*}
\mathcal{S}(A):=\left\{ (u_1,\ldots, u_k)\in \left(H^1(A)\right)^k:\ u_i\geq 0 \text{ in } A,\ u_i\cdot u_j=0 \text{ if } i\neq j \text{ and } -\Delta u_i\leq f_i(x,u_i), \right.\\
\left.-\Delta( u_i-\sum_{j\neq i} u_j)\geq f_i(x,u_i)-\sum_{j\neq i} f_j(x,u_j) \text{ in }A \text{ in the distributional sense}\right\},
\end{multline*}
where $f_i:\Omega \times [0,\infty)\rightarrow \R$ are $C^1$ functions with $f(x,s)=O(s)$ as $s\to 0^+$, uniformly in $x$. Recall the following result.

\begin{theorem}[{{\cite[Corollary 8.5]{TavaresTerracini}}}]\label{thm:ClassS}
Under the previous notation and assumptions, let $(u_1,\ldots, u_k)\in \mathcal{S}(A)$. Then each $u_i$ is locally Lipschitz continuous in $A$, and the conclusions of (2) Theorem \ref{thm:regularity} hold in $A$.
\end{theorem}

We recall that the part regarding Lipschitz continuity had already been proved in \cite{CTV_JFA03,CTV_Indiana05,CoTeVe05}.

 \begin{proposition}\label{prop:sclass-infty}
   Suppose that \eqref{eq:strict-p} holds for $p=\infty$ and, for $p$ large, let $(u_{1,p},\ldots, u_{k,p})$ be a minimizer for $\optenergyrel{k}{p}(\Omega)$, normalized in $L^2(\Omega)$.  For each $i=1,\ldots, k$, take
     \[
    a_{i,p}^2:=\frac{\mathcal{R}_V(u_i)^{p-1}}{\optenergyrel{k}{\infty} (\Omega)^{p-1}}\in \mathbb{R}\setminus \{0\}. 
     \]
     Then there exists $(\tilde u_1,\ldots, \tilde u_k)$, a minimizer for $\optenergyrel{k}{\infty}(\Omega)$, and $\tilde a_1,\ldots, \tilde a_k>0$  such that, up to a subsequence, as $p \to \infty$,
     \[
     u_{i,p}\to \tilde u_i \text{ strongly in } H^1_{0,V}(\Omega),\qquad a_{i,p}\to \tilde a_i,\qquad \text{ for $i=1,\ldots, k$}
     \]
     and, for $\tilde v_i=\tilde a_i \tilde u_i$, the following differential inequalities are satisfied in the distributional sense:
     \begin{enumerate}
     \item $-\Delta \tilde v_i + V(x) \tilde v_i \leq \mathcal{R}_V(\tilde v_i) \tilde v_i$;
     \item $-\Delta\left(\tilde v_i-\sum_{j \neq i} \tilde v_j\right) + V(x) \left(\tilde v_i-\sum_{j \neq i} \tilde v_j\right) \geq \mathcal{R}_V(\tilde v_i) \tilde v_i-\sum_{j \neq i} \mathcal{R}_V(\tilde v_j) \tilde v_j$.
     \end{enumerate}
      \end{proposition}
\begin{proof}
From Theorem~\ref{thm:existence} and 
\txtr{Proposition \ref{prop:p-continuity}}, for large $p$ and up to a subsequence, there exists $U_{p}=(u_{1,p}, \ldots, u_{k,p})$, a minimizer for $\optenergyrel{k}{p}(\Omega)$  such that 
$U_p \to U_\infty$ converges \txtr{weakly} in $H^1_{0,V}$ to a minimizer of $\optenergyrel{k}{\infty}(\Omega)$. Moreover, 
\begin{gather*}
    a_{i,p} = \frac{\mathcal{R}_V(u_{i,p})^{\tfrac{p-1}{2}}}{{\optenergyrel{k}{\infty}(\Omega)}^{\tfrac{p-1}{2}}} \le \left (\frac{\optenergyrel{k}{p}(\Omega)}{\optenergyrel{k}{\infty}(\Omega)}\right )^{\tfrac{p-1}{2}} \le k^{\tfrac{p-1}{2p}}\le \sqrt{k}
\end{gather*}
so, up to a subsequence, $a_{i,p} \to \tilde a_i$ as $p \to \infty$. \txtr{Then either
\begin{gather*}
    \mathcal{R}_V(u_{i,p}) \to \optenergyrel{k}{\infty}(\Omega) ,
\end{gather*}
(in which case $ u_{i,p}\to u_i$ strongly in $H^1_{0,V}$, and so also $v_{i,p}:=a_{i,p}u_{i,p}\to \tilde v_i$)
or else $a_{i,p} \to 0$ as $p\to \infty$ (and so $v_{i,p}\to 0=\tilde v_i$).  In particular, passing to the limit as $p\to\infty$ in the differential inequalities in Proposition~\ref{prop:diff-ineq-p} implies that $\tilde v_i$ will satisfy (1) and (2) for the nonzero components
}

\medbreak

We have
\begin{gather*}
    \sum_{i=1}^k \tilde a_i^2=\lim_{p \to \infty} \sum_{i=1}^k a_{i,p}^{2p/(p-1)}=\lim_{p \to \infty} \frac{\optenergyrel{k}{p}(\Omega)^p}{\optenergyrel{k}{\infty}(\Omega)^p} \ge 1,
\end{gather*}
so not all  $\tilde a_i$ are zero. To conclude, we check that  $\tilde a_i\neq 0$ for every $i$. Given $r>0$, the $k$-tuple with components $\tilde v_i|_{\Omega\cap B_r} = \tilde a_i \tilde u_i|_{\Omega\cap B_r}$ belongs to the set $\mathcal{S}(\Omega \cap B_r)$ for $f_i(x,s):=(\mathcal{R}_V(\tilde u_i)-V(x))s$. Since $f_i$ is of class $C^1$ (by our assumption on $V$) and $|f_i(x,s)|=O(s)$ as $s\to 0^+$ uniformly for $x\in \Omega\cap B_r$, then we can directly apply Theorem \ref{thm:ClassS}. This yields $|\Omega\cap B_r\cap \{x:\ \tilde v_i=0 \ \forall i\}|=0$ for every $r>0$, and so
\begin{gather*}
    \overline{\Omega} = \bigcup_{i\in \mathcal{I}} \overline{\{ \tilde v_{i} >0 \}},
\end{gather*}
where $\mathcal{I}=\{i:\ \tilde v_{i}\neq 0\}\neq\emptyset$. Assuming now that $\tilde a_1=0$, then $1\notin \mathcal{I}$ and $\{\tilde u_1>0\}$ (which is nonempty) is disjoint from $\bigcup_{i\in \mathcal{I}} \overline{\{ \tilde v_{i} >0 \}}=\overline \Omega$, a contradiction.
\end{proof}

\begin{remark}\label{rmk:allainotzero}
The above proof yields the existence of \emph{at least one} strong, regular partition. To have regularity for \emph{all} solutions associated with $\optenergyrel{k}{\infty}$ one would need to use other methods, such as an adaptation of the one in \cite{HHT09}. Similarly, for $N=2$ one can use the methods therein to further reduce the assumptions on the potential and obtain regularity up to the boundary, assuming more regularity of the latter.
\end{remark}


\begin{proof}[Proof of Theorem \ref{thm:regularity}]
 Let  $(u_1,\ldots, u_k)$ be any minimizer of $\optenergyrel{k}{p}(\Omega)$ (if $p\in [1,\infty)$); or let it be the minimizer of $\optenergyrel{k}{\infty}(\Omega)$ constructed in Proposition \ref{prop:sclass-infty} (if $p=\infty$). By Proposition \ref{prop:diff-ineq-p} (for $p<\infty$) and Proposition \ref{prop:sclass-infty} (for $p=\infty$), such $k$-tuple belongs to $\mathcal{S}(\Omega\cap B_r)$ for every $r>0$, for $f_i(x,s)=(\mathcal{R}_V(u_i)-V(x))s$. The conclusion now follows from Theorem \ref{thm:ClassS}.
\end{proof}

\begin{proof}[Proof of Proposition \ref{prop:weak-and-strong}]
A direct consequence of Theorem~\ref{thm:regularity}-(1) is that \eqref{eq:weak-and-strong} holds for $1\le p \le \infty$  if the weak threshold condition \eqref{eq:strict-p} does, since then a minimizer of $\optenergyrel{k}{p}(\Omega)$ exists by Theorem~\ref{thm:secondmain}. Indeed, under \eqref{eq:strict-p} and by Theorem~\ref{thm:secondmain}, we have the existence of a minimizer $(u_1,\ldots, u_k)$ of $\optenergyrel{k}{p}(\Omega)$ which is Lipschitz continuous  by Theorem \ref{thm:regularity}-(1). Then  $(\{u_1\neq 0\},\ldots, \{u_k\neq 0\})\in \partitionset{k}$ and this, combined with \eqref{eq:inequality_relaxed}, yields that it is a minimizing partition for $\optenergy{k}{p} (\Omega)$, whence $\optenergyrel{k}{p}(\Omega)= \optenergy{k}{p}(\Omega)$ (and $\thresholdrel{k}{p}=\threshold{k}{p}$).

In general, note that we always have the inequalities
\begin{equation}\label{eq:prop-weak-and-strong-help-2}
\optenergyrel{m}{p}(\Omega) \le \optenergy{m}{p}(\Omega) \le \threshold{m}{p}(\Omega)
\end{equation}
for all $1 \leq m \leq k$, by definition and by Theorem~\ref{thm:firstmain}, respectively. We will show by induction on $m$ that, in fact, $\thresholdrel{k}{p}(\Omega)=\threshold{k}{p}(\Omega)$ and $\optenergyrel{k}{p}(\Omega) = \optenergy{k}{p}(\Omega)$.

To this end, note that when $m=1$, the inequality $\thresholdrel{1}{p}(\Omega)=\threshold{1}{p}(\Omega)$ is immediate from the definition. Hence, either $\optenergyrel{1}{p}(\Omega) < \thresholdrel{1}{p}(\Omega)$ and the equality $\optenergyrel{1}{p}(\Omega) = \optenergy{1}{p}(\Omega)$ follows from the above argument, or else $\optenergyrel{1}{p}(\Omega) = \thresholdrel{1}{p}(\Omega) = \threshold{1}{p}(\Omega)$, whence $\optenergyrel{1}{p}(\Omega) = \optenergy{1}{p}(\Omega)$ by \eqref{eq:prop-weak-and-strong-help-2}.

Now suppose that we have both equalities for some $m \geq 1$, we consider $m+1$. It follows from the definition of the respective thresholds and the assumption that $\optenergyrel{m}{p}(\Omega) = \optenergy{m}{p}(\Omega)$ that in turn $\thresholdrel{m+1}{p}(\Omega) = \threshold{m+1}{p}(\Omega)$. We may now repeat verbatim the dichotomy argument used when $m=1$ to conclude that indeed $\optenergyrel{m+1}{p}(\Omega) = \optenergy{m+1}{p}(\Omega)$. This completes the proof of \eqref{eq:weak-and-strong}; the statement after it is then an immediate consequence of Theorem~\ref{thm:secondmain}, \eqref{eq:weak-and-strong}, and Theorem~\ref{thm:regularity}. So it remains to show (1) to (3).

For item (1), observe that, for the $p=\infty$ case and if \eqref{eq:strict-p} is not satisfied, then Theorem~\ref{thm:firstmain} guarantees the existence of a partition $\partition$ with $\energy{k}{\infty}(\partition) = \optenergy{k}{\infty}(\Omega)=\infess(\Omega)$. This shows that $\optenergy{k}{\infty}(\Omega)$ is always achieved, regardless of whether it is strictly less than $\infess (\Omega)$ or not, and regardless of whether $\optenergyrel{k}{\infty}(\Omega)$ is attained or not. 

For item (2), we refer to Examples~
\ref{ex:nopotential} and ~\ref{ex:strip}, while for item (3) we refer to Example~\ref{ex:fortschrittchen}.
\end{proof}

\begin{proof}[Proof of Proposition~\ref{prop:egalite}]
Let $\tilde v_i = \tilde a_i \tilde u_i$ be the minimizer of $\optenergyrel{k}{\infty}(\Omega)$ considered in Proposition \ref{prop:sclass-infty}, with $\tilde u_i$ being a weak limit of an $L^2$-normalized sequence $u_{i,p}$ of minimizers for $\optenergyrel{k}{p}(\Omega)$, and 
\begin{gather*}
    0 \neq \tilde a_i^2 = \lim_{p \to \infty} \tilde a_{i,p}^2=\lim_{p \to \infty} \left (\frac{\mathcal{R}_V(u_{i,p})}{\optenergyrel{k}{\infty}(\Omega)}\right )^{p-1}.
\end{gather*}
Observe that 
\begin{equation}\label{eq:lim_p_prop1.9}
\lim_{p\to \infty} \frac{\mathcal{R}_V(u_{i,p})}{\optenergyrel{k}{\infty}(\Omega)}=\frac{\mathcal{R}_V(\tilde u_{i})}{\optenergyrel{k}{\infty}(\Omega)}\in (0,1].
\end{equation}
If $\mathcal{R}_V(\tilde u_{i})<\optenergyrel{k}{\infty}(\Omega)$, then the limit in \eqref{eq:lim_p_prop1.9} belongs to the (open) interval $(0,1)$, and so
\[
\lim_{p\to \infty} \tilde a_{i,p}^2=0,
\]
a contradiction.

For item (2), see Example~\ref{ex:liberte}.
\end{proof}

\txtr{
We are now also in a position to give the complete proof of Proposition~\ref{prop:p}.}
\begin{proof}[Proof of Proposition~\ref{prop:p}]
\txtr{Proposition~\ref{prop:p}(1) and (2) follow from Proposition~\ref{prop:p-continuity} since $\optenergyrel{k}{p}(\Omega) = \optenergy{k}{p}(\Omega)$ for $1\le p \le \infty$ due to Proposition~\ref{prop:weak-and-strong}. Thus, it only remains to show (3):} By Theorem~\ref{thm:existence}, for each $p \in I$ there exists a minimizing $k$-tuple $U_p$, with $L^2$-normalized components, and
\begin{equation}
\energyrel{k}{p}(U_p) = \optenergyrel{k}{p}(\Omega) \to \optenergyrel{k}{p_0}(\Omega)
\end{equation}
when $p \to p_0$. From this it follows that, whenever $p_n \to p_0$, the corresponding sequence of minimizers $(U_{p_n})$ is actually a minimizing sequence for $\optenergyrel{k}{p_0} (\Omega)$, that is, for $p=p_0$. \txtr{For $1\le p_0 < \infty$ by Theorem~\ref{thm:existence}, it admits a subsequence convergent in $H_{0,V}^1(\Omega)$ to a minimizer of $\optenergyrel{k}{p_0}(\Omega)$. By Proposition~\ref{prop:sclass-infty} we have $\mathcal R_V(u_{i,p_n}) \to \optenergyrel{k}{\infty}(\Omega) $  for $p_n \to \infty$ and then similarly by Theorem~\ref{thm:existence} we infer that there exists a convergent subsequence $H_{0,V}^1(\Omega)$ to a minimizer of $\optenergyrel{k}{\infty}(\Omega)$.}
\end{proof}

\begin{proof}[Proof of Proposition~\ref{prop:aok}]  Monotonicity in $k$ of $\optenergy{k}{p}(\Omega)$ is immediate since, given any $(k+1)$-partition, unifying any two cells will produce a $k$-partition whose energy cannot be larger. Similarly, the threshold value is non-decreasing (in $k$) for $1<p<\infty$, as it is the  composition of the increasing function $X\mapsto (X^p + \Sigma^p)^{1/p}$ with $\kappa\mapsto \optenergy{k}{p}(\Omega)$.

Assume now that $\optenergy{k}{p}(\Omega)<\threshold{k}{p}(\Omega)$, for some $1\leq p\leq \infty$ and $k\in \N$, and let us prove that
\begin{equation}\label{eq:mon_in_k_aux}
\optenergy{k}{p}(\Omega)<\optenergy{k+1}{p}(\Omega).
\end{equation}
We have two cases:

\noindent Case 1. $\optenergy{k+1}{p}(\Omega)=\threshold{k+1}{p}(\Omega)$. In this situation,
\[
\optenergy{k}{p}(\Omega)<\threshold{k}{p}(\Omega)\leq \threshold{k+1}{p}(\Omega)= \optenergy{k+1}{p}(\Omega).
\]

\noindent Case 2. $\optenergy{k+1}{p}(\Omega)<\threshold{k+1}{p}(\Omega)$. Then by Theorem~\ref{thm:firstmain} there exists a minimizer $\partition=(\Omega_1, \Omega_2, \ldots, \Omega_{k+1})\in \partitionset{k+1}$ attaining $\optenergy{k+1}{p}(\Omega)$. 

Let us consider at first $1\le p < \infty$. Then
\begin{displaymath}
    \optenergy{k+1}{p}(\Omega) = \energy{k+1}{p}(\Omega_1, \ldots, \Omega_{k+1}) > \energy{k}{p}(\Omega_1, \ldots, \Omega_k) \ge \optenergy{k}{p}(\Omega)
\end{displaymath}
and we have $\optenergy{k}{p}(\Omega) < \optenergy{k+1}{p}(\Omega)$. 

Finally, for $p<\infty$, assuming \eqref{eq:mon_in_k_aux} and using the fact that the map $X\mapsto (X^p + \Sigma^p)^{1/p}$ is strictly increasing,  also $\threshold{k}{p}(\Omega)<\threshold{k+1}{p}(\Omega)$.
\end{proof}

\section{Examples}
\label{sec:examples}

This section is devoted to the examples listed and described in Section~\ref{sec:description-examples}.

\begin{example}
\label{ex:compact-resolvent}
We start with two simple examples which show how our main results (Theorems~\ref{thm:secondmain} and~\ref{thm:regularity}, and Proposition~\ref{prop:weak-and-strong}) generalize previous results on the existence of spectral minimal partitions (see \cite{BNHe17,CoTeVe05}, etc.), namely to cases where the underlying operator still has compact resolvent even though the domain is unbounded.

\begin{enumerate}
\item If we take $\Omega \subset \R^d$ to be an unbounded domain of finite volume (for example, append a finite number of infinite cusps, each of finite volume, to a bounded domain), and take $V = 0$, then  the embedding $H^1_0(\Omega) \hookrightarrow  L^2(\Omega)$ is still compact (for a necessary and sufficient condition on $\Omega$ that ensures compactness, see \cite[Theorem 6.16]{Adams_book}), the Dirichlet Laplacian on $\Omega$ still has compact resolvent and, in particular, $\infess (\Omega) = \infty$. By Theorem~\ref{thm:secondmain}, Theorem \ref{thm:regularity} and Proposition~\ref{prop:weak-and-strong}, $\optenergy{k}{p}(\Omega) = \optenergyrel{k}{p}(\Omega)$ admits a minimizing $k$-partition (in both the strong and weak senses) for all $1 \leq p \leq \infty$ and all $k \geq 1$.

\item The same conclusions hold if we take $\Omega = \R^d$ and $V(x) = |x|_2^2$ (corresponding to the quantum harmonic oscillator) since in this case it is well known (see, e.g., \cite[Section~2]{Fre18}) that $-\Delta + V$ has compact resolvent and thus, also in this case, $\infess (\R^d) = \infty$. More generally, this is known to be true whenever $\Omega \subset \R^d$ is any open set and $V(x) \to \infty$ as $|x|\to \infty$.
\end{enumerate}
\end{example}

We next give another family of simple examples where the threshold condition \eqref{eq:strict-p} is indeed satisfied for $p < \infty$ (for $p=\infty$ it is much easier to find such examples).

\begin{example}
\label{ex:threshold-yes}
Fix $k \ge 1$ and $p \in [1,\infty]$.  We give an example of a potential $V \in L^\infty (\R^d)$ for which \eqref{eq:strict-p} holds but the operator does not have compact resolvent. 

To this end, let $\Omega \subset \R^d$ be any bounded domain, and let $V_\Omega : \Omega \to \R$  be any nonnegative potential in $L^\infty(\Omega)$. We define a family of potentials $V_c$, $c > 0$, by
\begin{displaymath}
    V_c(x) := \begin{cases} V_\Omega (x) \qquad &\text{if } x \in \Omega,\\
    c\qquad &\text{otherwise.} \end{cases}
\end{displaymath}
It is standard in this case, and can be seen directly via \eqref{eq:from-persson-with-love}, that $\infess (\R^d,V_c) = c$. (In this example we will use a slightly different notation for our spectral threshold and minimal energies as we will need to compare minimal energies of different potentials.)

We claim that \eqref{eq:strict-p} holds for $V=V_c$ if $c>0$ is large enough. To see this, first note that
\begin{displaymath}
    \optenergyrel{k}{p}(\R^d,V_c) \le \optenergyrel{k}{p}(\Omega,V_\Omega),
\end{displaymath}
as follows from monotonicity of $\infspec$ with respect to domain inclusion, for any fixed potential (in other words, any partition of $\Omega$ may be extended arbitrarily to a partition of $\R^d$ with the same or lower energy). Now since $\optenergyrel{k}{p}(\Omega,V_\Omega)$ is a fixed number, we can clearly choose $c>0$ large enough such that
\begin{displaymath}
    \optenergyrel{k}{p}(\R^d,V_c) \le \optenergyrel{k}{p}(\Omega,V_\Omega) < c = \infess (\R^d,V_c).
\end{displaymath}
If $p=\infty$, this immediately yields \eqref{eq:strict-p}. If $p \in [1,\infty)$, take $p$-th powers and add $\optenergyrel{k-1}{p}(\R^d,V_c)^p \ge 0$ to the right-hand side to obtain the conclusion.
\end{example}

We next give a trivial example which shows that $\optenergy{k}{p}$ can be attained where $\optenergyrel{k}{p}$ is not.

\begin{example}\label{ex:nopotential}
    We take $\Omega= \mathbb R^d$ and $V=0$, then
    \begin{displaymath}
        \lambda(\Omega) = \Sigma(\Omega) =0
    \end{displaymath}
    and we have $\optenergyrel{k}{p}(\R^d)= \optenergy{k}{p}(\R^d)=0$ for all $k \geq 1$ and all $p \in [1,\infty]$. By Remark~\ref{rem:bounds}-(2), there exist spectral minimal partitions attaining $\optenergy{k}{p}(\R^d)=0$ for any $k \geq 1$ and $p\in [1,\infty]$. However, for any admissible $k$-tuple of functions $U=(u_1, \ldots, u_k)$ we have $\energyrel{k}{p} >0$, since otherwise $u_1=u_2=\cdots = u_k= 0$. Hence, there do not exist any spectral minimizers of $\optenergyrel{k}{p}(\R^d)$. 
\end{example}

Note that, in the previous example, the cells of the minimizing partition will not necessarily be connected, since in the construction used in Remark~\ref{rem:bounds}-(2) (from the proof of \eqref{eq:optimal-inequality-infty}) each ``cell'' will be an infinite union of annuli of increasing width. Our next example sketches how one can construct a minimizing partition with \emph{connected} cells (concretely, for the case $p=\infty$ and in dimension $d=2$, although the same construction also works for any $1 \le p<\infty$) in the prototypical case of the Laplacian ($V=0$) in an infinite strip in $\R^2$. Note however that, as in the previous example, no ground states exist.

\begin{example}
\label{ex:strip}
Suppose $V=0$. If $\Omega = (1,\infty) \times (0,\pi) \subset \mathbb R^2 $ is an infinite strip of width $\pi$, then it is a standard result that $\sigma(\Omega) = \sigma_{{\rm ess}}(\Omega) = [1,\infty)$ (see also Example \ref{ex:durchbruechchen} below). Moreover, by domain monotonicity, $\lambda(\Omega_1) \ge 1$ for all $\Omega_1 \subset \Omega$; combining this with \eqref{eq:optimal-inequality-infty}, we have $\optenergy{k}{\infty}(\Omega) =1$ for all $k\ge 1$.

If $k=1$ we may simply take $\Omega = \Omega_1$, while for $k\ge 2$ we can construct a minimal partition by constructing a series of ``rooms'' (rectangles) of increasing length, joined by increasingly narrow passages (cf.\ Figure~\ref{fig:geschachtelte-reihenhaeuser}). For $k=2$ we can, for example, take
\begin{align*}
    \Omega_1 &= \operatorname{Int}\left (\bigcup_{j\in \mathbb N} \left([2^{2j-2}, 2^{2j-1}] \times [0, \pi - \tfrac{1}{j}]\right) \cup \left([2^{2j-1}, 2^{2j}] \times [0, \tfrac{1}{j}]\right) \right ) \\
    \Omega_2 &= \operatorname{Int}\left (\bigcup_{j\in \mathbb N} \left([2^{2j-2}, 2^{2j-1}] \times [\pi- \tfrac{1}{j}, \pi ]\right) \cup \left([2^{2j-1}, 2^{2j}] \times [\tfrac{1}{j}, \pi ]\right) \right ) ;
\end{align*}
then $\infspec(\Omega_1) = \infspec(\Omega_2) = 1$, so that $\partition_2 := (\Omega_1,\Omega_2)$ will be a spectral minimal partition for all $1\le p \le \infty$. (To see that $\infspec(\Omega_2) = 1$, say, denote by $R_j := [2^{2j-1},2^{2j}] \times [\frac{1}{j},\pi]$ the $j$-th ``room'' in $\Omega_2$; then for each $j \in \N$, by domain monotonicity $\infspec(\Omega_2) \leq \infspec (R_j) = (\frac{\pi}{\pi-\frac{1}{j}})^2 + (\frac{\pi}{2^{2j}-2^{2j-1}})^2 \to 1$ as $j \to \infty$). However, $\Omega_1$ and $\Omega_2$ do not admit ground states.
\begin{figure}[ht]
\centering
\hfill
\begin{minipage}{.45\textwidth}
\includegraphics[scale=0.25]{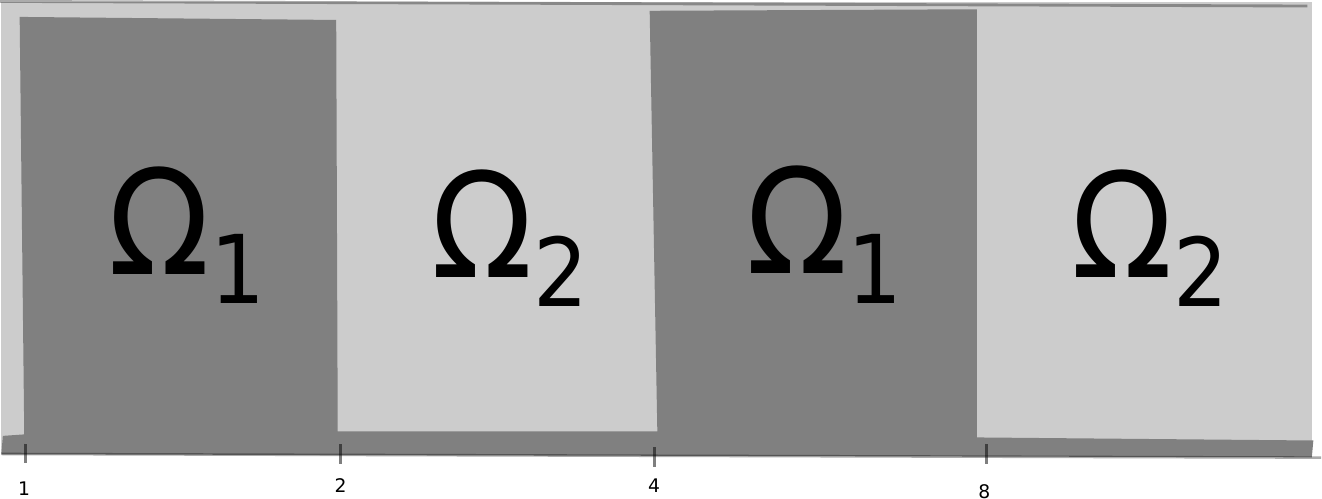}
\end{minipage} \hfill \begin{minipage}{.45\textwidth}
\includegraphics[scale=0.25]{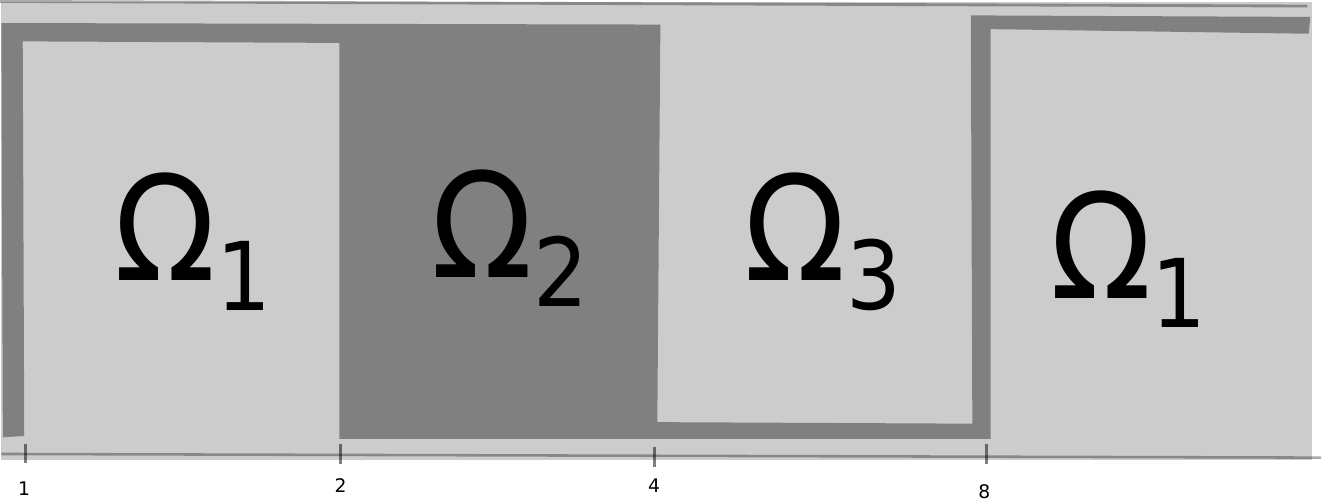}
\end{minipage}
\caption{A schematic representation of the construction of connected minimal partitions of the strip $(1,\infty) \times (0,\pi)$  for $k=2$ (left) and $k=3$ (right) (not to scale on the horizontal axis: the length of the ``rooms'' is chosen to increase towards infinity away from the left endpoint $x=1$). We introduce increasingly narrow passages to connect the rectangular regions (which will be of increasing length as $x$ increases). In order to treat the case $k\ge 4$ the domain $\Omega_2$ can be divided additionally (for $k=4$ indicated by the dotted line, note that of course the domains need to be rescaled).}
\label{fig:geschachtelte-reihenhaeuser}
\end{figure}

For any $k \geq 3$ it is still possible to find an optimal partition, that is, some $\partition_k = (\Omega_1,\ldots,\Omega_k)$ such that $\infspec (\Omega_i) = \infess (\Omega_i) = 1$ for all $i=1,\ldots,k$; but, again, in this case no $\Omega_i$ will have a ground state. To give an example to illustrate the principle, for simplicity we take $k=3$.

More precisely, for all $j \geq 1$ we consider the rectangle
\begin{displaymath}
    S_j :=
    \begin{cases}
    [2^{j-1},2^j-\tfrac{1}{3j}] \times [0,\pi - \tfrac{1}{j}]\qquad &\text{if } j \equiv 1 \mod 3,\\
    [2^{j-1}- \tfrac{1}{3j},2^j - \tfrac{2}{3j}] \times [\tfrac{1}{2j},\pi-\tfrac{1}{2j}]\qquad &\text{if } j \equiv 2 \mod 3 \\ 
    [2^{j-1}- \tfrac{2}{3j},2^j-\tfrac{1}{j}] \times [\tfrac{1}{j},\pi] \qquad &\text{if } j \equiv 0 \mod 3,
    \end{cases}
\end{displaymath}
(cf.\ Figure~\ref{fig:geschachtelte-reihenhaeuser}). We then form $\Omega_1$ by connecting all $S_j$ where $j=3\ell-2$ for some $\ell \geq 1$, that is, we take
\begin{displaymath}
    \Omega_1 = \operatorname{Int}\left (\bigcup_{\ell \in \N} S_{3\ell-2} \cup \bigcup_{\ell \in \N} [2^{3\ell-2}- \tfrac{1}{3j},2^{3\ell-1}- \tfrac{2}{3j}] \times [0,\tfrac{1}{j}] \cup \bigcup_{\ell \in \N} [2^{3\ell-1}- \tfrac{2}{3j},2^{3\ell}- \tfrac{1}{j}] \times [0,\tfrac{1}{j}] \right ), 
\end{displaymath}
and analogously for $\Omega_2$ (with $j=3\ell-1$) and $\Omega_3$ (with $j=3\ell$). By construction, the $\Omega_i$ are disjoint and, by domain monotonicity, as before,
\begin{displaymath}
    1 \leq \infspec (\Omega_1) \leq
    \infspec (S_j)=  \frac{\pi^2}{(\pi - \frac{1}{j})^2} + \frac{\pi^2}{(2^{j-1}- \tfrac{1}{3j})^2} \to 1
\end{displaymath}
as $j=3\ell - 2 \to\infty$, meaning that $\infspec (\Omega_1) = 1$ (similarly for $\Omega_2$ and $\Omega_3$).

It is immediate that an analogous construction would work for any $k \geq 4$. Note that the example becomes even easier if we do not insist that the $\Omega_i$ be \emph{connected} sets, since then we do not need to arrange the thin connecting ``passages'' running along the edge of the strip. Likewise, such an example can easily be extended to higher dimensions, see the next remark.
\end{example}

\begin{remark}
\label{rem:super-strip}
    More generally, let $\Omega \subset \mathbb R^d$ with $d\ge 3$ be an open (unbounded) domain, without loss of generality with $0\in \Omega$. We may also assume that $\Omega$ is equipped with a potential $V \in L^\infty_{\text{loc}} (\Omega)$.

    If we do not insist that the $\omega_i$ in a spectral minimal partition be \emph{connected} sets, then we can \emph{always} find a minimizing partition for $\optenergy{k}{\infty}(\Omega)$. Indeed, if $\optenergy{k}{\infty}(\Omega) < \infess(\Omega)$, then we already know that there exists a minimizer (each of whose cells is in fact connected), by Theorem~\ref{thm:secondmain}. Otherwise, if $\optenergy{k}{\infty}(\Omega) = \infess(\Omega)$, then we may take the $\omega_i$ to be the unions of concentric annuli used in the proof of Theorem~\ref{thm:firstmain}, see \eqref{eq:olympic-rings}: in this case, by construction, in the notation of that proof,
    \begin{displaymath}
        \infspec (\omega_i) \leq \lim_{j \to \infty} \infspec (\Omega \cap A_{r_{jk+i},R_{jk+i}}) = \infess (\Omega)
    \end{displaymath}
    for each $i=1,\ldots,k$, so that the corresponding partition has energy $\infess(\Omega)$.
    
    For $d\ge 3$, we can always make each $\omega_i$ connected using a version of the narrow passages used in the model Example~\ref{ex:strip}, namely using so called ``fireman's poles''. 
    Fixing any $i=1,\ldots,k$, We choose any smooth connecting paths $p_j^{(i)} \subset \Omega$, contained in $\overline{A_{R_{jk+i},r_{jk+i+1}}}$, which will connect $\Omega \cap A_{r_{jk+i},R_{jk+i}}$ with $\Omega \cap A_{jk+i+1},R_{jk+i+1}$. We then consider a sequence of radii $\rho_j^{(i)}$ such that for each $i$ the (Choquet) capacities (cf.\ \cite{RaTa75} or \cite[Chapter~2]{HKM06}) of the tubular neighborhoods $B_{\rho_j^{(i)}}(p_j^{(i)}) = \{ x \in \Omega: \dist(x,p_j^{(i)})<\rho_j^{(i)} \}$ are summable in $j$, i.e. 
    \begin{displaymath}
        \sum_{j\in \mathbb N} \operatorname{cap} (B_{\rho_j^{(i)}}( p_j^{(i)})) < \infty.
    \end{displaymath}
    We then define
    \begin{equation}
    F_{\rho^{(i)}}:= \bigcup_{j\in \mathbb N} B_{\rho_j^{(i)}}( p_j^{(i)}) = \bigcup_{n\in \mathbb N} \{ x\in \Omega : \dist(x, p_j^{(i)}) < \rho_j^{(i)}\}.
    \end{equation}
    Using a cutoff argument (which we omit) one can show
    \begin{align*}
        \lambda\left (\Omega_i\cup F_{\rho^{(i)}} \setminus \operatorname{clos} \left (\bigcup_{\hat i\neq i} F_{\rho^{(\hat i)}}\right )\right ) &\le  \lim_{N\to \infty} \lambda\left ( \Omega_i  \setminus \left ( \operatorname{clos} \left ( \bigcup_{\hat i \neq i} F_{\rho^{(i)}} \right )\cap (\Omega \setminus B_{r_N}(0)) \right ) \right ) \\
        &= \lim_{N\to \infty} \lambda( \Omega_i \cap B_{r_N}(0)), 
    \end{align*}
    since by construction
    \begin{equation}
        \lim_{N\to \infty} \operatorname{cap}\left (  \operatorname{clos} \left ( \bigcup_{\hat i \neq i} F_{\rho^{(i)}} \right )\cap (\Omega \setminus B_{r_N}(0)) \right ) =0
    \end{equation}
    (cf.\ the proof of \cite[Theorem~2.43]{HKM06}; here, for legibility, we have used $\operatorname{clos}$ for the closure of a set).
    In dimension $d=2$ such an argument cannot be used directly, since curves do not have capacity zero in $\R^2$. It seems likely that one could run the joining passages along the boundary of $\Omega$, as in Example~\ref{ex:strip}, but it would be subtle to show that such a construction is always possible; we do not go into further details.
\end{remark}

Even though in such cases we can find always a minimizing partition for $\optenergy{k}{\infty}(\Omega)$, we can use the principle explained in Remark~\ref{rmk:clrestimate} to construct an example of such a minimizing partition which is not an equipartition (see Proposition~\ref{prop:egalite}).

\begin{figure}[ht]
\centering
\includegraphics[scale=0.6]{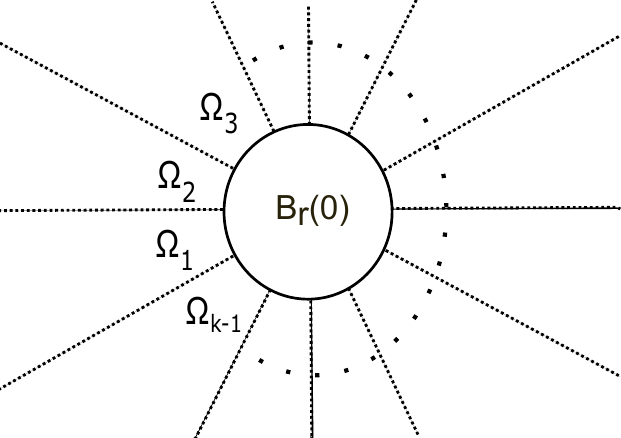}
\caption{An illustration of how for any given $k$, one can choose a $k$-partition consisting of $k-1$ wedges together with a central ball.}
\label{fig:watermelon_slices}
\end{figure}

\begin{example}
\label{ex:liberte}
For $d \geq 2$ and given constants $r,c>0$, we consider the potential $V_{r,c} : \R^d \to \R$ given by
\begin{displaymath}
    V_{r,c}(x) = \begin{cases}
    0 \qquad &\text{if $|x|_2 < r$,}\\ c \qquad &\text{otherwise.}
    \end{cases}
\end{displaymath}
In this case it is well known that $\infess = c$ (indeed, outside a compact set we have $-\Delta + c$ on $\R^d$). We fix $r>0$ and $c>0$ large enough that $\lambda_1(B_r(0),V) < c$. Now standard results yield that
\begin{displaymath}
    \infess (\R^d) = c,
\end{displaymath}
and $-\Delta + V_{r,c}$ has only finitely many eigenvalues below $c$ (indeed, the latter assertion follows from \eqref{eq:CLR}). Hence, by \eqref{eq:energy-eigenvalue}, there exists some $k \geq 2$ such that, for this potential $V$,
\begin{displaymath}
    \optenergy{k}{\infty} (\R^d) = \infess (\R^d) = c.
\end{displaymath}
If we take a partition $\partition = (\Omega_1,\ldots,\Omega_{k-1},B_r(0))$ where each $\Omega_i$ is the intersection of $\R^d \setminus B_r(0)$ with a suitable sector (if $d=2$ we may take a sector of angle $\frac{2\pi}{k-1}$, cf.\ Figure~\ref{fig:watermelon_slices}; if $d\geq 3$ this is easily generalized by taking wedges via sectors of a single angular coordinate), then in particular $V=c$ identically in $\Omega_i$.

In this case we can show by a direct comparison that $\infspec (\Omega_i) = c$. Indeed, on the one hand, it is immediate that $\infspec (\Omega_i) \geq c$, since $V=c$ on $\Omega_i$ and the Laplacian is positive; on the other, for any radius $R>0$, we can find a ball $B_R$ of radius $R$ contained in $\Omega_i$; thus, by the variational characterization, $\infspec (\Omega_i) \leq \lambda_1 (B_R)$. Since $\lambda_1 (B_R) \to c$ when $R \to \infty$, the claim follows.

We thus have $\infspec (\Omega_i) = c$ for all $i=1,\ldots,k-1$, while $\infspec (B_r(0)) = \lambda_1 (B_r(0)) < c$. Thus $\partition$ is not an equipartition, but $\energy{k}{\infty} (\partition) = c$, so it is a minimizing $k$-partition.
\end{example}

\begin{figure}[ht]
\centering 
\hfill 
\begin{minipage}{.45\textwidth}
\includegraphics[scale=0.5, trim= 0 0 1cm 0, clip]{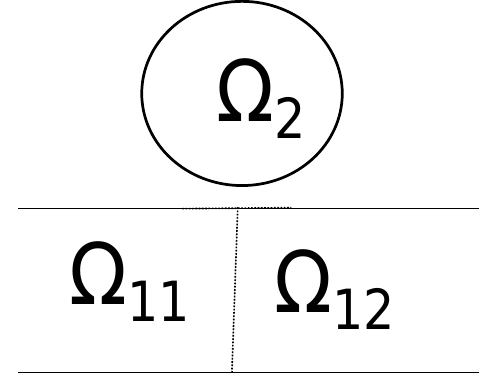}
\end{minipage} \begin{minipage}{.45\textwidth}
\includegraphics[scale=0.5, trim= 0 0 1cm 0, clip]{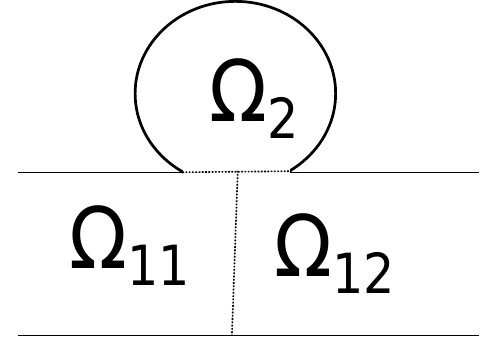}
\end{minipage}
\caption{A representation of the domains in Example~\ref{ex:strip-ball} and the corresponding minimizing three-partitions described there. These form equipartitions for which not all of the partition elements admit ground states.}
\label{fig:stripballtopconnected}
\end{figure}

\begin{example}
\label{ex:strip-ball} 
We give examples where there exists a minimizing partition for $\optenergy{k}{p}$ (concretely, with $k=3$ and any $p\in [1,\infty]$) which is an equipartition, but where not all of the partition elements admit a ground state.

Let $V\equiv 0$. We first consider disconnected domains, taking the following. Let $\Omega_1 := \R \times (0,\pi)^{d-1}$ be an infinite strip in $\R^d$, which has $\infspec (\Omega_1) = \infess (\Omega_1) = 1$ (similar to the strip in Example~\ref{ex:strip}), and let $\Omega_2$ be a disjoint ball in $\R^d$ such that $\infspec (\Omega_2) = 1$ as well. Then it is immediate  that $\infspec (\Omega) = \infess (\Omega) = 1$ for $\Omega := \Omega_1 \sqcup \Omega_2$, and $\optenergy{k}{p}(\Omega)=1$.

Now, for any $1 \leq p \leq \infty$, there is an optimal $3$-partition $\partition$ of $\Omega := \Omega_1 \sqcup \Omega_2$ of energy $1$, which is thus necessarily a minimizing partition since $\infspec(\Omega)=1$ also: take $\Omega_{11}$ and $\Omega_{12}$ to be two equal half-strips which partition $\Omega_1$, and take $\partition := \{ \Omega_{11}, \Omega_{12}, \Omega_2\}$. We see that $\Omega_2$ admits a ground state, but $\Omega_{11}$ and $\Omega_{12}$ do not.

If we do not permit disconnected domains, we can achieve the same result in essentially the same way, at least for $p=\infty$, although we will not go into full details. We glue $\Omega_1$ and $\Omega_2$ to form a connected domain $\Omega$ (see Figure~\ref{fig:stripballtopconnected}-right, and note that $\Omega_2$ has been deformed in such a way that it meets $\Omega_1$ along a surface and not just a point, but $\infspec(\Omega_2) = 1$ still; one may if one wishes take $\Omega_2$ to be a cube of unit side length instead). We claim that $\partition$ as described above is still a minimizing partition for $\optenergy{3}{\infty} (\Omega)$. Note that $\Omega$ will have exactly one eigenvalue below $\infess (\Omega)$, since $\infspec (\Omega) < \infspec (\Omega_2) = 1 = \infess (\Omega)$; however, one can show that $\Omega$ only has this one eigenvalue below $\infess (\Omega) = 1$. By \eqref{eq:counting-bounds}, necessarily $\optenergy{3}{\infty} (\Omega) = \optenergy{2}{\infty} (\Omega) = 1$.

Note that in both cases the minimizing $3$-partition is not unique, and that we can find other minimizing $3$-partitions for which \emph{no} partition element admits a ground state (just use the same idea as in Example~\ref{ex:strip}, excluding $\Omega_2$). Also note that both these examples can be easily generalized to $k$-partitions for $k \geq 4$. Finally, we observe that it should not be possible in this case to have a minimizing $3$-partition for $\optenergy{3}{p} (\Omega)$ for any $1 \leq p < \infty$, although we will not show that here; such an example will be given in Example~\ref{ex:fortschrittchen}.
\end{example}

\begin{remark} Take $\Omega$ to be the same as in the previous example; then, keeping the same notation, $\Omega_1$ (the strip) and $\Omega_2$ (the ball) are both minimizers for $\optenergy{1}{p}$. In this case, $\Omega_1$ does not have a ground state, while $\Omega_2$ does. This gives a trivial example of a domain with non-unique spectral minimal partitions for some $k$, where in one minimizer all cells admit a ground state, while in another no cell has a ground state.
\end{remark}

For our final example we start by finding a domain and a potential where we have a predetermined number of eigenvalues below $\infess$; this will be useful to construct partitions with certain properties afterwards. We thus split the example up into two parts: the construction in Example~\ref{ex:durchbruechchen}, and the description of the properties of such domains in Example~\ref{ex:fortschrittchen}.

\begin{example}
\label{ex:durchbruechchen}
Fix any $m \in \N$. We give an example of a domain $\Omega\subset \R^2$ and a potential $V=V(x,y)$ for which the corresponding Schrödinger operator on $\Omega$ has \emph{exactly} $m$ eigenvalues below $\infess$. We take an infinite half-strip $\Omega=(0, \infty) \times (0,\ell \pi) \subset \mathbb R^2$, where $\ell>0$, and write $(x,y) \in \R^2$. Consider the (one-dimensional) potential
\begin{displaymath}
    V_L(x) = \begin{cases}
    	c, &\quad x> L,\\
            0, &\quad x \le L.
    \end{cases}
\end{displaymath}
for $c>0$ and $L>0$. Then the Dirichlet eigenvalues of 
\begin{displaymath}
 - u'' + V_L(x) u = \lambda u \quad \text{ in } (0,\infty)
\end{displaymath}
are characterized as solutions of the transcendental equation
\begin{equation}\label{eq:transcendantal}
     \tan(\sqrt{\lambda} \, L) = - \frac{1}{\sqrt{c/\lambda-1}}, 
\end{equation}
where $\omega = \sqrt{\lambda} >0$. 

On the strip we separate variables: we have $L^2((0, \infty) \times (0, \ell\pi)) = \overline{L^2(0, \infty) \otimes L^2(0, \ell\pi)}$, and hence, writing $u(x,y) = u_1(x) u_2(y)$, since the potential $V(x,y)=V_L(x)$ only depends on $x$ we have
\begin{gather*}
-\Delta + V(x,y) |_{L^2(0,\ell \pi) \otimes L^2(0,\infty)} = \left (-\frac{\mathrm d^2}{\mathrm dx^2} + V_L(x) \right ) \otimes \mathbb 1 + \mathbb 1 \otimes \left ( - \frac{\mathrm d^2}{\mathrm dy^2}\right ).
\end{gather*}

In particular, the eigenvalues of $-\Delta + V_L$ are the sums of eigenvalues $\lambda_1, \lambda_2$ of the eigenvalue problems
\begin{displaymath}
 \begin{aligned}
    -u_1 ''(x) + V_L(x) &= \lambda_1 u_1(x) \\
    - u_2''(y)  &= \lambda_2 u_2(y).
 \end{aligned}
\end{displaymath}

Now if we choose $\frac{\pi^2}{4L^2} < c < \frac{\pi^2}{L^2},$ then there is exactly one solution $\lambda = \lambda_0 \in (\frac{\pi^2}{4L^2}, \frac{\pi^2}{L^2})$ of \eqref{eq:transcendantal}. Then for the point spectrum
\begin{displaymath}
\sigma_p(-\Delta + V_L) = \sigma_p\left (-\frac{\mathrm d^2}{\mathrm dx^2} + V_L\right ) + \sigma_p\left (-\frac{\mathrm d^2}{\mathrm dy^2} \right ) =\{ \lambda_0 + k^2\ell^{-2}: k\in \mathbb N \},
\end{displaymath}
while we have $\sigma_\text{ess}(-\Delta + V_L) = \sigma_{\text{ess}}(-\tfrac{\mathrm d^2}{\mathrm dx^2} + V_L) + \sigma(-\tfrac{\mathrm d^2}{\mathrm dy^2}) = [c + \ell^{-2},\infty)$, and in particular
$$\infess(\Omega) = c + \ell^{-2}.$$
In particular, given $m \in \N$, if we take $\ell > 0$ such that
\begin{displaymath}
     \frac{m^2-1}{\ell^2}< c - \lambda_0 < \frac{(m+1)^2-1}{\ell^2},
\end{displaymath}
then we will have
\begin{displaymath}
    \lambda_0 + \frac{m^2}{\ell^2} < c + \frac{1}{\ell^2} = \infess (\Omega) < \lambda_0 + \frac{(m+1)^2}{\ell^2},
\end{displaymath}
and there are exactly $m$ eigenvalues below the infimum of the essential spectrum.
\end{example}

\begin{figure}[ht]
\centering
\includegraphics[scale=0.4]{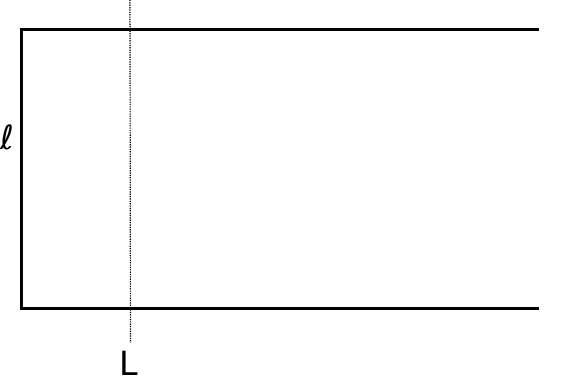}
\caption{The half-strip from Example~\ref{ex:fortschrittchen}. The potential is taken to be $0$ to the left of $L$, and $c$ to its right.}
\label{fig:half-strip}
\end{figure}

\begin{example}\label{ex:fortschrittchen}
Let $1\le p < \infty$. By Example~\ref{ex:durchbruechchen}, we can find an operator with precisely one eigenvalue below the infimum of the essential spectrum and an embedded (i.e. not isolated in $\sigma(\Omega)$) second eigenvalue equal to $\infess(\Omega)$. More precisely, keeping the notation from there, and setting $\tau_k (\Omega) := \lambda_0 + k^2 \ell^{-2}$, $k \in \N$, to be the (possibly embedded) eigenvalues, if we choose $\ell^2 \ge 3(c-\lambda_0)^{-1}$, then as seen in Remark~\ref{rmk:clrestimate} we have
\begin{displaymath}
\optenergy{2}{\infty} (\Omega) \geq \lambda_2(\Omega)= \infess(\Omega) 
\end{displaymath}
On the other hand with Theorem~\ref{thm:firstmain} we infer $\optenergy{2}{\infty}(\Omega) =  \Sigma(\Omega)$ and thus
\begin{displaymath}
 \optenergy{2}{\infty} (\Omega)= c+ \ell^{-2} \le  \lambda_0 + 4\ell^{-2} = \tau_2 (\Omega).
\end{displaymath}
Then for any $ 2$-partition $(\Omega_1, \Omega_2)$ we have
\begin{displaymath}
    \energy{2}{\infty}((\Omega_1, \Omega_2)) =\max\{ \infspec(\Omega_1), \infspec(\Omega_2)\} \ge \optenergy{2}{\infty}(\Omega) = \infess(\Omega) = c + \ell^{-2}. 
\end{displaymath}
In particular, after a renumbering if necessary, we may assume $\infspec(\Omega_1) \ge \infess(\Omega)$, and since $\Omega_2 \subsetneq \Omega$ and $\infspec(\Omega)$ is a discrete eigenvalue, we have strict domain monotonicity, $\infspec(\Omega_2) > \infspec(\Omega) = \optenergy{1}{p}(\Omega)$. Thus for any given $2$-partition $(\Omega_1,\Omega_2)$ of $\Omega$, we have
\begin{displaymath}
    \energy{2}{p}((\Omega_1, \Omega_2))^p = \infspec(\Omega_1)^p + \infspec(\Omega_2)^p > \infess(\Omega)^p + \optenergy{1}{p}(\Omega)^p.
\end{displaymath}
However, by Theorem~\ref{thm:firstmain}, we also have $\optenergy{2}{p}(\Omega)^p \le \infess(\Omega)^p + \optenergy{1}{p}(\Omega)^p$. Hence no spectral minimal partition exists in this case.

To summarize, in this example:
\begin{enumerate} 
    \item $\optenergy{2}{p}(\Omega)$ and $\optenergyrel{2}{p}(\Omega)$ never admit minimizers for $1 \leq p < \infty$, as noted above.
    \item $\optenergy{2}{\infty}(\Omega)= \optenergyrel{2}{\infty}(\Omega)$ is \emph{always} attained (see Proposition~\ref{prop:weak-and-strong} and Example~\ref{ex:strip}).
    \item By varying our choice of $\ell$, we can produce two cases: 
    \begin{itemize}
        \item $\tau_2(\Omega) = \lambda_0 + 4\ell^{-2} \leq \infess(\Omega)$, and $\optenergyrel{2}{\infty}(\Omega)$ is attained;
        \item $\tau_2(\Omega) > \infess(\Omega)$, and in particular this eigenvalue will be embedded and $\optenergyrel{2}{\infty}(\Omega)$ is not attained.
    \end{itemize} 
    
    To see this, first note that in either case any eigenfunction for $\tau_2$ will still have exactly two nodal domains, as follows directly from the product structure of the Laplacian (sketched in the previous example) and the fact that the ground state on $(0,\infty)$ is positive while, on $(0,\ell\pi)$, Sturm's oscillation theorem holds.
    
    Now, if $\tau_2 (\Omega) \leq \infess(\Omega)$, then the two nodal domains of the corresponding eigenfunction yield a test partition of energy smaller than or equal to $\infess(\Omega)$, so that either \eqref{eq:strict-p} holds for $p=\infty$ and a minimizer exists, or $\optenergyrel{2}{\infty}(\Omega) = \infess(\Omega)$ and thus the nodal partition for $\tau_2(\Omega)$ is optimal.
    
    On the other hand, if $\tau_2 (\Omega) > \infess(\Omega)$, then, for a proof by contradiction, suppose there exists a minimizing $2$-tuple $(u_1,u_2)$ for $\optenergyrel{2}{\infty}(\Omega)$. Then we can find a linear combination of $u_1$ and $u_2$ orthogonal to the first eigenfunction to use as a trial function in the Rayleigh--Ritz principle and hence show the existence of a second (variational) eigenvalue less than or equal to $\optenergyrel{2}{\infty}(\Omega)$. This is a contradiction to the assumption that $\tau_2(\Omega) > \infess (\Omega) = \optenergyrel{2}{\infty}(\Omega)$, and that the $\tau_k$ exhaust the point spectrum of $\Omega$.
\end{enumerate}
It should be possible to do the same for $k \geq 3$ in place of $k=2$ if $\Omega$ is adjusted to have exactly $k-1$ eigenvalues strictly less than $\infess(\Omega)$. We do not go into details.
\end{example}

\appendix

\bibliographystyle{plain}

\end{document}